\documentclass{amsart_m}

\usepackage[dvips]{epsfig}
\usepackage{graphicx}
\usepackage{float}

\usepackage{ulem}

\usepackage{amsmath}
\usepackage{amssymb}
\usepackage{amsbsy}
\usepackage{bbm}
\usepackage{eucal}
\usepackage{mathrsfs}
\usepackage{mathabx}

\usepackage{enumitem}

\usepackage[colorlinks=true,citecolor=blue, linkcolor=red]{hyperref}

\usepackage{dsfont}
\usepackage{palatino}
\usepackage{xcolor}

\usepackage{multirow}

\usepackage{amsthm}

\newtheorem{theorem}{Theorem}

\newtheorem{assumption}{Assumption}
\newtheorem{lemma}[theorem]{Lemma}

\newtheorem{definition}[theorem]{Definition}
\newtheorem{proposition}[theorem]{Proposition}
\newtheorem{remark}[theorem]{Remark}

 \newcommand{\trash}[1]{}
 \newcommand{\triple}{|\!|\!|}

\usepackage{amscd}

\usepackage{tensor}

\newcommand{\opnorm}[1]{{\vert\kern-0.25ex\vert\kern-0.25ex\vert #1 
  \vert\kern-0.25ex\vert\kern-0.25ex\vert}}

\newcommand{\R}{\mathbb{R}}
\newcommand{\p}{\mathbb{P}}
\newcommand{\E}{\mathbb{E}}

\newcommand{\bs}[1]{\boldsymbol{#1}}

\makeatletter
\newcommand{\leqnomode}{\tagsleft@true}
\newcommand{\reqnomode}{\tagsleft@false}
\makeatother

\def\widebarroman#1{\sbox0{#1}\dimen0=\dimexpr\wd0+1pt\relax
  \makebox[\dimen0]{\rlap{\vrule width\dimen0 height 0.06ex depth 0.06ex}%
    \rlap{\vrule width\dimen0 height\dimexpr\ht0+0.03ex\relax 
            depth\dimexpr-\ht0+0.09ex\relax}%
    \kern.5pt#1\kern.5pt}}

\allowdisplaybreaks

\usepackage{setspace}

\begin{document}

\date{\today}
\author[F. Delarue\: ; \: R. Maillet\: ; \: E. Tanré]
{François Delarue, \quad 
Raphaël Maillet, \quad
Etienne Tanré }

\address{François Delarue, Universit\'e C\^ote d'Azur, Laboratoire J.A.Dieudonn\'e, 06108 Nice, France}
\email{francois.delarue@univ-cotedazur.fr}
\address{Raphaël Maillet, Universit\'e Paris-Dauphine \& PSL, CNRS, CEREMADE, 75016 Paris, France}
\email{maillet@ceremade.dauphine.fr}
\address{Etienne Tanr\'e, Université Côte D’Azur, Inria, CNRS, LJAD, France}
\email{Etienne.Tanre@inria.fr}

\title[Ergodicity of some stochastic Fokker-Planck equations]{Ergodicity of some stochastic Fokker-Planck equations with additive common noise}

\begin{abstract} 
{In this paper we consider stochastic Fokker-Planck Partial Differential Equations (PDEs), obtained as the mean-field limit 
(i.e., as the number of particles tends to $\infty$)
of weakly interacting   particle systems 
subjected to both independent (or idiosyncratic) and common Brownian noises. 
We provide sufficient conditions under which the deterministic counterpart of the Fokker-Planck equation, 
which corresponds to particle systems that are just subjected to independent noises, 
has several invariant measures, but for which the stochastic version admits a unique 
invariant measure under the presence of the additive common noise. 
The very difficulty comes from the fact that the common noise is just of finite dimension while 
the state variable, which should be seen as the conditional marginal law of the system given the common noise, lives in 
a space of infinite dimension. In this context, our result holds true 
if, in addition to standard confining properties, the mean field interaction term forces the system to be attracted by its conditional 
mean given the common noise and the intensity of the idiosyncratic noise is small.}
\end{abstract}

\maketitle

\noindent \textbf{Mathematics Subject Classification (2020)}: 37A50, 35Q70, 60H15, 60J25

\noindent \textbf{Keywords}: 
{Non-linear Fokker-Planck equations , McKean-Vlasov equations with common noise, Stochastic Partial Differential Equations, Common noise, Asymptotic stability}

\tableofcontents

\section{Introduction}
This paper addresses the long-term behavior of solutions to a class of non-linear Stochastic Partial Differential Equations (SPDEs) of the form
\begin{equation}\label{SMKV}
	\mathrm{d}_t  m_t = \nabla \cdot \left(\frac{\sigma^2 + \sigma_0^2}{2} \nabla m_t - m_t b(\cdot, m_t)\right)\mathrm{d}t -\sigma_0 \nabla m_t \cdot \mathrm{d}B^0_t, \quad \textrm{\rm on } [0, +\infty)\times\R^d,
\end{equation}
where $B^0=(B_t^0)_{t \geq 0}$ is a $d$-dimensional Brownian motion, called \textit{common noise}. 
Above, the unknown $(m_t)_{t \geq 0}$ is regarded as a stochastic process with values in the space 
${\mathcal P}({\mathbb R}^d)$
of probability measures  on 
${\mathbb R}^d$. Accordingly, the coefficient $b$ is a function from $\R^d\times\mathcal{P}\left(\R^d\right)$ to $\R^d$ 
depending on both space and measure arguments. The two parameters $\sigma$ and $\sigma_0$ are non-negative scalars: the former
is the intensity of the so-called idiosyncratic noise and the latter is the intensity of the common noise. 

Equation \eqref{SMKV} is in fact intended to describe the flow of conditional marginal laws of the solution to the conditional McKean-Vlasov equation:
\begin{equation}\label{SMKV-X}
	\mathrm{d}  X_t = b\bigl(X_t,{\mathcal L}(X_t\vert B^0)\bigr) \mathrm{d}t +
\sigma  \mathrm{d}B_t
+
	\sigma_0 \nabla m_t \cdot \mathrm{d}B^0_t, \quad t \geq 0,
\end{equation}
where $B$ is another Brownian motion, independent of $B^0$, and ${\mathcal L}(\cdot \vert B^0)$ stands for the conditional 
law given the realization of the common noise. Formally, $m_t$ in 
\eqref{SMKV} is expected to coincide with ${\mathcal L}(X_t\vert B^0)$ in \eqref{SMKV-X}. The rigorous connection between 
\eqref{SMKV} and \eqref{SMKV-X} is addressed more carefully in the core of the paper.

\subsection{Deterministic Fokker-Planck equations} When $\sigma_0=0$, i.e., in absence of common noise, 
Equation  
\eqref{SMKV} boils down to a standard Fokker-Planck equation describing the evolution of the marginal laws 
of the solutions to the 
standard McKean-Vlasov equation \eqref{SMKV-X} obtained by replacing ${\mathcal L}(X_t\vert B^0)$
by the law of $X_t$. 
In this setting, the long-time analysis of the solutions has been a notoriously challenging problem for over twenty years, giving rise to numerous contributions, many of them still recent. While the problem is of mathematical interest in its own right, 
it also finds some application to calculus of variation since certain equations like 
\eqref{SMKV} can be interpreted as gradient flows on the space of probability measures, see the seminal works 
\cite{ambrosio2005gradient,carrillo2003kinetic,JKO,Otto1,Otto2}. 
For instance, such a gradient structure occurs when 
 $b(x,\mu) = -\nabla V(x) -\nabla W \ast\mu(x)$, for two differentiable real-valued functions $V$ and $W$ 
 on ${\mathbb R}^d$, $W$ being symmetric, where $\ast$ stands for the convolution operator. In this case, the potential lying above the dynamics writes  
\begin{equation}
\label{eq:potential:MKV} 
 \mu \in {\mathcal P}(\R^d) \mapsto \int_{\R^d} V(x)\mu({\mathrm d}x) + \frac12 \int_{\R^d} 
 \int_{\R^d} W(x-y) \mu({\mathrm d}x) \mu({\mathrm d}y) + \frac{\sigma^2}2 \int_{\R^d} 
 \ln \bigl( \frac{{\mathrm d} \mu}{{\mathrm d} x}(x) \bigr) 
 \mu({\mathrm d}x).
 \end{equation} 

In fact, due to the mean field interaction (whether it derives from a potential or not), 
 the sole presence of the noise $B$ 
in the dynamics \eqref{SMKV-X} (with $\sigma_0=0$) does not suffice to guarantee the uniqueness of an invariant measure
(or equivalently of a stationary solution to \eqref{SMKV}) under  quite general 
conditions on $b$. This is in contrast with the long run behavior of
 Stochastic Differential Equations (SDEs) (or 
equivalently linear Fokker-Planck equations), 
for which mere confining properties of the drift suffice to produce ergodicity (provided $\sigma >0$). 
In the mean field setting, some further structural conditions are necessary. In the potential example 
\eqref{eq:potential:MKV}, a standard condition is to demand $V$ and $W$ to be (strongly) convex, 
which forces in fact a form of convexity of the potential 
\eqref{eq:potential:MKV} on the space of probability measures. We refer for instance to  
the earlier works
\cite{benachour1998nonlinear,carrillo2003kinetic,malrieu2001logarithmic,malrieu2003}, which also include 
additional convergence results to the stationary regime. 
\textcolor{black}{For a tiny list of variants, in which \textit{perturbations} of the convex potential case are addressed, we refer to 
\cite{bashiri2020long,bolley2013uniform,butkovsky2014ergPropNLMCandMKV,delmoral2019uniform,durmus2020elementary}.
We refrain from detailing all the possible extensions of the previous cases under which 
the invariant measure remains unique. Let us just say, for our purposes, that the above list of 
references includes cases where the lack of convexity is compensated for by the presence of a strong enough diffusion coefficient.
This \textit{large noise} regime 
is compared with our own setting in Remark 
\ref{rem:choice:parameters} below. At this point, we would especially like to stress} that there are known 
 explicit simple cases in which uniqueness does not hold, which observation supports our previous claim: the presence of a nonlinear mean field term may \textit{easily} lead to the existence of multiple invariant measures, even in situations where the mass stays confined under the dynamics \eqref{SMKV}. Specifically, it has been proven in \cite{herrmann2010non} 
 (see also \cite{{1983JSP....31...29D}})
 that, when $V$ is only uniformly convex outside of a ball  but admits a double-well, $W$ is quadratic (but non-zero) and the diffusion coefficient $\sigma$ is sufficiently small, there exist several stationary solutions 
 to 
\eqref{SMKV}
 when it derives from the potential \eqref{eq:potential:MKV}
  (and $\sigma_0=0$). 

\subsection{Common noise}
Very basically, our primary objective in this article is to revisit the class of examples addressed in 
\cite{herrmann2010non} , but in presence of a common noise, i.e., $\sigma_0>0$. In this regard, it is worth noting that,
when 
  $b(x,\mu) = -\nabla V(x) -\nabla W \ast\mu(x)$, with $V$ being uniformly convex and $W$ being convex, 
  the common noise does not change the picture recalled in the previous paragraph. Indeed, the second author has 
  shown in \cite{maillet2023note} that \eqref{SMKV} has a unique invariant measure, 
  which is consistent with the results obtained earlier without common noise with the slight subtlety that the invariant measure is then  understood 
  as a probability measure on ${\mathcal P}(\R^d)$.
  
  Of course, our goal is to go one step beyond and prove that the common noise can change the picture positively, 
  meaning that it can force uniqueness of the invariant measure even though 
  the deterministic analogue of \eqref{SMKV} (i.e., with 
  $\sigma_0=0$) has several stationary solutions.
  Actually, this type of result, if it holds true, must be part of the very broad theory 
  of ergodic Markov processes, with the specific feature that the state space here is
  ${\mathcal P}(\R^d)$. It is worth observing that, from the same point of view,  
    the deterministic and linear analogue of \eqref{SMKV-X} (i.e., $b$ only depends on 
    $x$ and $\sigma_0=0$) should be regarded as a Fokker-Planck equation obtained by forcing a standard ordinary
    differential equation by a Brownian motion of intensity $\sigma$. This is just to say that, in the Euclidean setting, 
    the same program is very well-understood and just consists in addressing the ergodic properties of 
    a non-degenerate SDE. 
   For sure, the very difficulty in our setting comes from the fact that the unknown in 
    \eqref{SMKV} lives in a space of infinite dimension (once again, the space of probability measures)
    while the noise (i.e., $B^0$) is just finite-dimensional and thus \textit{completely} degenerate. This says that, at best, 
    noise can be expected to restore uniqueness of the invariant measure only in specific cases.
     And, precisely, this is our objective to identify a class of cases that would become ergodic under the action of the common noise and that would include
some of the examples addressed in \cite{herrmann2010non}. 
    
    In fact, this question was already addressed by the second author in 
     \cite{maillet2023note} when $\sigma=0$, leaving open setting with $\sigma >0$. Apart from 
     this, the idea of using the averaging properties of the common noise has been used in several contexts, for instance in the analysis 
     of existence and uniqueness of equilibria to mean field games (see for instance \cite{delarue2020selection,Foguen}
     for games  with a finite dimensional common noise and  \cite{delarue2019restoring}
     for games with an infinite dimensional common noise). More recently, the first author has just released 
     an article, \cite{delarue2024rearranged}, in which the ergodic properties of 
     \eqref{SMKV} (or, more precisely, \eqref{SMKV-X}) are studied when $b$ is general but the common noise
     is infinite dimensional and the dimension $d$ is 1. Part of the challenge then precisely lies in the construction of the noise, which is a question 
     different from the one addressed here. A similar problem was addressed in the prior work \cite{angeli2023mckean}, for another type of infinite dimensional common noise that may not preserve the space of probability measures.

\subsection{Our contribution}
We here focus on drifts of the type 
\begin{equation*}
	b(x,\mu) = G(x) + F\left(x-\int_{\R^d} x \: \mu(\mathrm{d}x)\right), \quad x \in {\mathbb R}^d, \ \mu \in {\mathcal P}(\R^d), 
\end{equation*}
where $F$ and $G$ are functions from $\R^d$ into itself. The function $F$ is assumed to be strictly decreasing, 
which occurs for instance if $F=-\nabla W$ for a strictly convex function. More interestingly, $G$ is just assumed to be 
confining in a mild sense that is detailed in Assumption~\ref{A1} below. In particular, $G$ is not required
to be strictly decreasing and, when it derives from a potential $V$, $V$ may not be convex.  
Therefore, this framework allows us to choose $G= -\nabla V$ for a non-convex potential $V$ and $F = - \nabla W$, for $W(x) = \alpha x^2$ for some $\alpha > 0$,
and thus contains the case treated by  \cite{herrmann2010non} (up to an additional common noise), since 
$b(x,\mu)$  rewrites $-\nabla V(x) -\nabla W \ast\mu(x)$. 

Our main contribution is to establish that, for a small diffusion coefficient $\sigma$ and for an interaction force $F$ that is sufficiently decreasing, the system has a unique invariant measure, which attracts exponentially fast any other initial condition (with appropriate integrability properties). Our 
strategy of proof is based on the key idea that the long run behavior of \eqref{SMKV-X} should be dictated by the 
long time dynamics of the conditional mean ${\mathbb E}[X_t \vert B^0]$ given the common noise. However, this intuition is not completely correct, as 
the distance between
$X_t$ and its conditional mean ${\mathbb E}[X_t \vert B^0]$  does not tend to $0$ in long time, at least when 
$\sigma >0$. In fact, some residual fluctuation persists due to the presence of the idiosyncratic noise and this is one of our 
main achievement here to explain how to handle this residual fluctuation in long time. This is precisely the point where we need 
$\sigma$ to be small and the interaction force $F$ to be sufficiently decreasing. As for the function $G$, it mainly plays a role 
in the long time analysis of the conditional mean ${\mathbb E}[X_t \vert B^0]$ itself, which is shown to 
share many similarities with 
the long time analysis of the (standard) SDE driven by the drift $G$ and the noise $\sigma_0 B^0$. 
In particular, we draw heavily on previous works on coupling methods for long time analysis of SDEs, especially the approach developed in \cite{eberle2016reflection} on couplings by reflection (inspired from \cite{lindvall1986coupling}), which plays a key role in our study.

Notice that our result complements the earlier one \cite{maillet2023note}
dedicated to the case $\sigma=0$. 

	\subsection{Further connections with the literature on mean field models}
Stochastic PDEs
like \eqref{SMKV}
and related conditional McKean-Vlasov equations of 
type 
\eqref{SMKV-X} were introduced in 
\cite{dawson1995stochastic,kurtz1999particle}. 
Further, they were studied in a series of works 
\cite{lions2013scalar,lions2014scalar, gess2014scalar, friz2016stochastic, gess2017stochastic, fehrman2019well}
in connection with stochastic scalar conservations laws, in which case
$\sigma_0$ is typically required to depend on the local value at point $x$ of ${\mathrm d} m_t/{\mathrm d}x$.
The more recent contribution 
\cite{coghi2019stochastic} addresses uniqueness to \eqref{SMKV} under weaker conditions than in \cite{kurtz1999particle}.
Equation \eqref{SMKV} has also become very popular in 
mean field game theory, see for instance the book \cite{cardaliaguet2019master}. 
A variant, including a reflection term, has been studied recently in  \cite{briand2020forward}.  
As for \eqref{SMKV-X}, we refer to \cite{hammersley2021weak} for a general existence and uniqueness result of weak solutions. 

Last but not least, it is worth recalling that the connection between 
nonlinear Fokker-Planck equations and McKean-Vlasov equations (without common noise) goes back to the pioneering 
works of  \cite{kac1956foundations,mckean1966class,funaki1984certain}. 
In this context, a key question concerns the particular approximation of the solutions, usually referred to 
as \textit{propagation of chaos}, see for instance 
 \cite{sznitman1991topics,meleard_1996}. In connection with 
 the existence of stationary solutions to the Fokker-Planck equation 
 \eqref{SMKV}, it is in general a difficult question to wonder whether 
 propagation of chaos holds uniformly in time or not. We feel better not to give a list of  
 references in this direction but we quote that the question appears for instance in the earlier contribution 
 \cite{cattiaux2008probabilistic}. Obviously, this would be very interesting to address the same 
 problem but in presence of a common noise for the model studied here. 
%
	
	\subsection{Organisation of the paper}
	In Section~\ref{sec:main}, we list the assumptions and give the main statements of the paper. 
	We also address some examples and compare in particular our results with those from \cite{herrmann2010non}
	(when $\sigma_0=0$ and uniqueness does not hold). The proof of the main theorem (Theorem~\ref{th:main}) is split in two parts. 
	A first step is to derive Theorem~\ref{th:main} from a key auxiliary estimate (Proposition 
	\ref{prop:p1})  on the contraction properties of the 
	semigroup induced by the solution of \eqref{SMKV}. 
	This is done in Section~\ref{sec:th1}.
	A second step, which is in fact the core of the paper, is to prove this key estimate. 
	We do so in Section~\ref{proofP1}, using coupling arguments. Auxiliary results are proven in Appendix.
	

\color{black}
	\subsection{Notation}\label{sec:notation}
	
Throughout the paper, for a Polish space $E$, $\mathcal{P}(E)$ stands for the space of Borel probability measures on $E$ equipped with the topology of weak convergence and the corresponding Borel $\sigma$-algebra {(which is induced by the 
mappings $m \in {\mathcal P}(E) \mapsto m(A)$ for any Borel subset $A$ of $E$)}.
{Whenever there exists a distance $d$ so that $(E,d)$ is a metric space, 
we call ${\mathcal P}_p(E)$, for any $p>0$, the collection of elements $\mu \in {\mathcal P}_p(E)$ such that 
\begin{equation*}
	\exists x_0 \in E : \int_E d(x_0,x)^p \mu(\mathrm{d}x) < +\infty.
\end{equation*}
In fact, the integral above is finite or not, whatever the choice of $x_0$. We then} define for any $p\geq 1$, the $p$-Wasserstein distance on 
$\mathcal{P}_p(E)$ as
\begin{equation*}
	\mathrm{d}_p^E(\mu, \nu) := \inf_{\pi \in \Pi(\mu, \nu)} \left(\int_{E\times E} d(x,y)^p \: \pi(\mathrm{d}x, \mathrm{d}y)\right)^{1/p},
	\quad \mu,\nu \in {\mathcal P}_1(E), 
\end{equation*}
where $\Pi(\mu, \nu)$ stands for the set of all couplings of $\mu$ and $\nu$ (i.e., all the joint
probability measures on $E \times E$ with $\mu$ and $\nu$ as first and second marginals). 

Also, for any {measurable $\varphi : E \rightarrow {\mathbb R}$} and any probability measure $m \in \mathcal{P}(E)$, we define the duality product $\langle\cdot\: ;\cdot\rangle_E$ as 
\begin{align*}
    \langle m ; \varphi\rangle_E = \int_{E} \varphi(x) m(\mathrm{d}x),
\end{align*} 
{whenever the integral in the right-hand side does exit (for instance so is the case if 
$\varphi$ is bounded or $\varphi$ is at most of linear growth and $m \in {\mathcal P}_1(E)$ and so forth...).}
{When $E= \R^d$, we define}
\begin{align*}
	& \mu_1(m) := \int_{\R^d} x \: m(\mathrm{d}x), \quad {m \in {\mathcal P}_1(\R^d)}; \\
	& \mu_2(m) :=  \int_{\R^d} |x|^2 \: m(\mathrm{d}x),  \quad {m \in {\mathcal P}_2(\R^d)}; \\ 
	& v(m) := \int_{\R^d} \left| x - \mu_1(m)\right|^2\: m(\mathrm{d}x),
	 \quad {m \in {\mathcal P}_2(\R^d)}; 
\end{align*}
which stand respectively for the expectation, the moment of order two and the variance of $m$. It is important to pay attention to the fact that despite the similarities in the notation, the objects $\mu_1(m) \in \R^d$ and $\mu_2(m) \in \R$ are of different {dimensions}. 

{When $m$ is random, say is a measurable mapping from 
$(\Omega_0, \mathcal{F}^0, \mathbb{P}_0)$ into ${\mathcal P}(E)$ (or  ${\mathcal P}_1(E)$), its law $P$ is an element of
${\mathcal P}({\mathcal P}(E))$ (or ${\mathcal P}({\mathcal P}_1(E))$), in the sense that, for any bounded measurable function $\phi : \mathcal{P}(E) \to \R$, $\: \E_{0}[\phi(m)] = \langle P ; \phi   \rangle_{\mathcal{P}(E)}$
(and similarly, when working on  $\mathcal{P}_1(E)$), where
$\E_0$ is the expectation associated with ${\mathbb P}_0$. 
Repeatedly in this paper, we also consider continuous stochastic processes} $(m_t)_{t\geq 0}$, typically constructed on 
$(\Omega_0, \mathcal{F}^0, {\mathbb F}^0,\mathbb{P}_0)$, 
 with values in the space {$\mathcal{P}_1\left(\R^d\right)$ equipped with $\mathrm{d}_1^{\R^d}$ (and thus also in 
${\mathcal P}(\R^d)$, equipped with the weak convergence topology)}. 
{When $(m_t)_{t \geq 0}$ is Markovian, we define its semigroup $(\mathscr{P}_t)_{t\geq 0}$ by letting, for any function $\phi : \mathcal{P}(\R^d) \to \R$ and  any $t \geq 0$,} 
\begin{equation*}
	\mathscr{P}_t\phi : \
	\begin{array}{l} \mathcal{P}(\R^d) \to \R
	\\ m \mapsto \E_0[\phi(m_t) | m_0 = m] . 
	\end{array}
\end{equation*}

Moreover, for any twice differentiable function $f : \R^d \to \R$,  $\nabla f$, $\Delta f$ and $\nabla ^2 f$ respectively stand for the gradient, {Laplacian} and Hessian matrix of $f$. For any differentiable function $F : \R^d \to \R^d$, we denote by $DF : \R^d \to \mathcal{M}_d(\R)$ its Jacobian matrix.
 
{For any real squared matrix $A \in \mathcal{M}_d(\R)$, we denote the Euclidian norm of $A$ by 
	\[\triple A\triple = \sqrt{\text{Tr}(A^\top A)},\] 
		with $A^\top$ and $\text{Tr}$ respectively standing for the transpose matrix of $A$ and the Trace operator on $\mathcal{M}_d(\R)$. 
	Lastly, the Euclidean norm on $\R^d$ is denoted $\vert \cdot \vert$ and the inner product $\cdot$. }


\section{Main result: uniqueness recovery thanks to common noise}\label{sec:main}

\subsection{Stochastic Fokker-Planck PDE and Conditional McKean-Vlasov SDE} 
In this paper, we {focus on a mean field model with common noise} in which the nonlinearity in the dynamics 
{occurs} through the {mean (i.e., the expectation). To make it clear,
on a filtered probability space 
$\left(\Omega_0, \mathcal{F}^0, \mathbb{F}^0, \mathbb{P}_0\right)$ (satisfying the usual conditions) equipped with a $d$-dimensional $\mathbb{F}^0$-Brownian 
motion $
B^0$ and for 
$F,G : \R^d \to \R^d$ two interaction and confinement forces,} we are interested in the solutions of the following Stochastic Partial Differential Equation (SPDE):
\begin{equation}\label{eq:SFKP:sigma0}
	\mathrm{d}_t  m_t = \nabla \cdot \left(\frac{\sigma^2 + \sigma_0^2}{2} \nabla m_t - m_t \left( G + F(\cdot - \mu_1(m_t))\right)\right)\mathrm{d}t -\sigma_0 \nabla m_t \cdot \mathrm{d}B^0_t, \quad t \geq 0, 
\end{equation}
{which is understood the weak sense, i.e.,}  the process $(m_t)_{t\geq 0}$ satisfies for all $t \geq 0$ and $\varphi \in C_c^{\infty}(\mathbb{R}^d)$,
\begin{align*}
\mathrm{d}\langle m_t, \varphi\rangle=\langle m_t, {\mathscr L}_{m_t} \varphi \rangle \: \mathrm{d} t + \sigma_0\langle m_t,(\nabla \varphi)^{\top}\rangle \: \mathrm{d} B^0_t, 
\end{align*}
where for any probability measure $m \in {\mathcal{P}_1}(\R^d)$, the operator {${\mathscr L}_{m}$} acts on a smooth function $\varphi$ with compact support in the following manner:
\[
{\mathscr L}_{m} \varphi = - \left(G  + F(\cdot - \mu_1(m) \right) \cdot \nabla \varphi+\frac{\sigma_0^2+\sigma^2}{2} \Delta \varphi.
\]
Equation \eqref{eq:SFKP:sigma0} admits a Lagrangian representation in the form of conditional McKean-Vlasov equation. Throughout the article, 
this McKean-Vlasov equation is constructed on a product probability space, obtained by tensorizing  
the filtered probability space 
$(\Omega_0, \mathcal{F}^0, \mathbb{F}^0, \p_0)$
(on which Equation 
\eqref{eq:SFKP:sigma0} is defined) with another filtered probability space $(\Omega_1, \mathcal{F}^1, \mathbb{F}^1, \p_1)$. The product structure is denoted 
\begin{align*}
	(\Omega := \Omega_0\times \Omega_1, \:\: \mathcal{F}, \:\: \mathbb{F}, \:\: \p),
\end{align*}
where  $(\mathcal{F}, \p)$ is the completion of $(\mathcal{F}^0 \otimes \mathcal{F}^1, \p_0\otimes \p_1)$ and $\mathbb{F}$ is the right continuous augmentation of $( \mathcal{F}^0_t \otimes \mathcal{F}^1_t)_{t \geq 0}$. In this paper we denote by $\E_0$, $\E_1$ and $\E$ the expectations on $(\Omega_0, \mathcal{F}^0, \mathbb{F}^0, \p_0)$, $(\Omega_1, \mathcal{F}^1, \mathbb{F}^1, \p_1)$ and $(\Omega, \mathcal{F}, \mathbb{F}, \p)$ respectively. Analogously, for any random variable $X$ defined on $(\Omega, \mathcal{F}, \mathbb{F}, \p)$, we denote by $\mathcal{L}_0(X)$ and $\mathcal{L}_1(X)$ the conditional laws of the random variable $X$ given $\mathcal{F}^1$ and $\mathcal{F}^0$ respectively. Of course, $\mathcal{L}(X)$ stands for the law of $X$ on the whole product space. 

Next, 
the initial law of 
\eqref{eq:SFKP:sigma0} may be random and then distributed according to a probability measure 
$P_0 \in \mathcal{P}(\mathcal{P}_1(\R^d))$. 
The initial condition
$m_0$ is then defined as an $\mathcal{F}^0_0$-measurable random variable with values in 
$\mathcal{P}_1(\R^d)$ such that $\mathcal{L}(m_0) = P_0$. We can now define on the whole probability space $(\Omega, \mathcal{F}, \mathbb{F}, \mathbb{P})$ a random variable $X_0$ such that $\mathcal{L}_1( X_0) = m_0$ almost surely. Precisely, 
by Lemma 2.4 in \cite{CD2018_2}, for $\mathbb{P}^0$-a.e. $\omega_0\in\Omega_0$, $X_0(\omega_0, \cdot)$ is a random variable on $(\Omega_1, \mathcal{F}^1, \mathbb{P}_1)$ and the conditional law of $X_0$ given $\mathcal{F}^0$, $\mathcal{L}_1 : \Omega_0 \ni \omega_0 \mapsto \mathcal{L}(X(\omega_0, \cdot))$, defines a random variable from $(\Omega_0, \mathcal{F}^0, \mathbb{P}_0)$ into $\mathcal{P}(\R^d)$, which can be taken equal to $m_0$ almost surely. 

In addition to $B^0$, we consider a $d$-dimensional ${\mathbb F}^1$-Brownian motion $B$ supported by
the space $(\Omega_1, \mathcal{F}^1, \p_1)$. The Lagrangian formulation of 
\eqref{eq:SFKP:sigma0} then reads in the form of the following conditional McKean-Vlasov equation, 
set on $(\Omega,{\mathcal F},{\mathbb F},{\mathbb P})$:
\begin{equation}\label{LP}
	\mathrm{d}X_t = G(X_t)\mathrm{d}t + F\left(X_t - \E_1[X_t]\right) \mathrm{d}t + \sigma \mathrm{d}B_t + \sigma_0 \mathrm{d} B^0_t, 
	\quad t \geq 0,
\end{equation}
with $X_0$ as initial condition. 
\color{black}

	\subsection{Assumptions, existence and uniqueness of solutions}
In the sequel, we make the following assumptions on {the coefficients $F$ and $G$. We start with $G$}:
\begin{assumption}\label{A1}
	The drift $G : \R^d \to \R^d$ is differentiable in $\R^d$. Moreover, 
	\begin{itemize}
		\item[{\rm (A1.1)}] $G$ is confining in the sense that there exists a function $\kappa : [0; +\infty) \to \R$ 
		 {such that} 
		\begin{equation*}
{\forall x,y \in {\mathbb R}^d,} \quad 
			\left(G(x) - G(y)\right)\cdot(x-y) \leq -\frac{\sigma_0^2}{2}\kappa(|x-y|)|x-y|^2, 
		\end{equation*} 
		{with}
		\begin{equation*}
		\lim\sup_{r \to +\infty}\kappa(r) > 0 \quad 
		\textrm{\rm and} \quad 
			\int_0^1 r \kappa(r)^{-} \mathrm{d} r<\infty, 
		\end{equation*}
		where $\kappa^{-} = \max(0, -\kappa)$. 
		\item[{\rm (A1.2)}] $G$ is Lipschitz-continuous on $\R^d$, meaning that there exists a 
		constant $L_G > 0$ such that  
		\begin{equation*}
		{\forall x,y \in {\mathbb R}^d,} \quad	|G(x) - G(y)| \leq L_G|x-y|. 
		\end{equation*}
		\item[{\rm (A1.3)}] $G$ is differentiable and its Jacobian matrix $DG$ is Lipschitz-continuous on $\R^d$, meaning that there exists a constant $C_G > 0$ such that for any $x,y \in \R^d$, 
		\begin{equation*}
			\triple DG(x) - DG(y)\triple \leq C_G|x-y|.
		\end{equation*}
	\end{itemize}
\end{assumption}

	{In order to state properly the assumption on $F$}, we introduce the following definition:
\begin{definition}\label{D1}
	For a pair $(\alpha, C) \in \R\times [0, +\infty)$, we call $\mathcal{S}(\alpha, C, L)$ the collection of $L$-Lipschitz-continuous 
	and differentiable functions $H$
	{from $\R^d$ into itself that are $\alpha$-decreasing and whose derivative is $C$ 
	Lipschitz-continuous, i.e.}, for all $(x,y) \in \R^{2d}$,
	\begin{itemize}
		\item[{(i)}] $(x-y) \cdot (H(x)  - H(y)) \leq -\alpha |x-y|^2$;
		\item[{(ii)}] $\triple DH(x) - DH(y)\triple \leq C |x-y|.$
	\end{itemize}
\end{definition}
{Given Definition~\ref{D1}}, we make the following assumption on $F$:
\begin{assumption}\label{A2}
	$F : \R^d \to \R^d$ satisfies $F(0) =0$. Moreover, there exist $\alpha_F > 0$ and $C_F \geq 0$ such that $ F\in \mathcal{S}(\alpha_F, C_F, C_F)$. 
\end{assumption}

\subsubsection*{About Assumption~\ref{A1}} The following remarks about Assumption~\ref{A1} are in order:
\begin{itemize}
\item Assumption~\ref{A1} is quite intuitive and consistent with the ideas presented in \cite{herrmann2010non}. Essentially, this assumption guarantees that \(G\) is confining and behaves linearly at infinity. Also, it is worth noting that our approach can accommodate a force that derives from a non-convex potential, which, around the origin, behaves for example like \(x \mapsto |x|^4 - |x|^2\).
\vskip 4pt

\item When $\sigma_0 > 0$, there exists a canonical choice of $\kappa$, given by
\begin{equation*}
	\kappa(r) {:=} \inf \left\{-\frac{2}{\sigma_0^2}\frac{(x-y) \cdot (G(x)-G(y))}{|x-y|^2} {;} \quad x, y \in \mathbb{R}^d \text { s.t. } {\vert x-y\vert}=r\right\}, \quad {r >0}. 
\end{equation*}
Assumption~\ref{A1} says that $\kappa$ is {necessarily} positive outside of a ball. This implies in particular that $\sigma_0^2\kappa^-/2$ is bounded from below by some constant $m_G \in \R$, which gives   
\begin{equation}\label{eq:lowerboundkappa}
	(x-y) \cdot (G(x)-G(y)) \leq -m_G|x-y|^2, \quad {x,y \in {\mathbb R}^d}. 
\end{equation}
This implies in particular that $G \in \mathcal{S}(m_G, C_G, L_G)$. 
\vskip 4pt 

\noindent When $m_G > 0$ we say that the drift is {(strictly) decreasing}. This case is easier to study under similar conditions on $F$ because the monotone structure of $G$ ensures that two solutions of the SDE \eqref{LP} (with different initial conditions) get closer in large time.  In this context, uniqueness of the invariant measure of the solutions of Equation \eqref{eq:SFKP:sigma0} is covered in \cite[Section 2]{maillet2023note}, {but it is fair to say that the common noise then plays little role because the invariant measure is also unique when 
$\sigma_0=0$. Notice that in the specific case when $G$ derives from a potential, 
this potential is strictly convex under the condition $m_G>0$.}
\vskip 4pt
\end{itemize}

\subsubsection*{About Assumption~\ref{A2}} Assumption ~\ref{A2} is somewhat surprising, 
{especially as we are asking below the coefficient $\alpha_F$ to
be large (see the statement of Theorem 
\ref{th:main}). Intuitively, one might indeed} expect that uniqueness of the invariant measure becomes less likely if the interaction is strong.
{At least, this is exactly what happens in absence of common noise, 
the extreme case being $F \equiv 0$ in which the dynamics become a mere diffusion equation
(the long time analysis of which is much easier to study).}
 However, {the situation is different here, thanks to the peculiar structure of the dynamics
 and to the presence of the common noise ($\sigma_0 > 0$). In brief, $F$ has a contracting effect, which forces 
 the process to be attracted by its conditional expectation (given the common noise), with the latter being a nearly solution of an ergodic SDE. Even though this picture is not exactly correct due to some minor errors caused by the idiosyncratic noise, 
 this is indeed a key step in our analysis to quantify the accuracy of this approximation and to derive from it uniqueness of the invariant measure.} We refer to Section~\ref{subsec:31}.
\vskip 4pt

{Throughout, we work under Assumptions~\ref{A1} and~\ref{A2}. In this context, we claim
\begin{proposition} 
\label{prop:existence:!}
Under Assumptions~\ref{A1} and~\ref{A2} and for an initial condition $X_0$ satisfying 
${\mathbb E}[\vert X_0 \vert]< \infty$ (or equivalently 
${\mathbb E}_0 \int_{{\mathbb R}^d} \vert x \vert m_0(\mathrm{d}x) < \infty$
for $m_0:= {\mathcal L}_0(X_0)$), 
the conditional McKean-Vlasov equation 
\eqref{LP} has a unique 
${\mathbb F}$-progressively measurable solution
$(X_t)_{t \geq 0}$, with continuous trajectories, such that, for 
all $T>0$, ${\mathbb E} [\sup_{0 \le t \le T} \vert X_t\vert] < \infty$. 
\end{proposition}}

{\begin{proof}
The proof consists in a straightforward fixed point argument, using the 
Lipschitz properties of the two coefficients $F$ and $G$. We refer to~\cite[Chap. 2]{CD2018_2}. 
\end{proof}}

{By the superposition principle for conditional McKean-Vlasov equations (see 
\cite{lacker2022superposition}), we deduce that existence and uniqueness also hold true for the 
stochastic Fokker-Planck equation: 
\begin{proposition}
Under Assumptions~\ref{A1} and~\ref{A2} and for an initial condition $m_0$ satisfying  
${\mathbb E}_0 \int_{{\mathbb R}^d} \vert x \vert m_0({\mathrm d}x) < \infty$, 
the stochastic Fokker-Planck equation 
\eqref{eq:SFKP:sigma0} has a unique solution $(m_t)_{t \geq 0}$
in the space of ${\mathbb F}^0$-progressively measurable processes 
with values in ${\mathcal P}_1(\R^d)$ satisfying 
${\mathbb E}_0[\sup_{0 \leq t \leq T} \int_{{\mathbb R}^d} \vert x \vert 
m_t(
{\mathrm d}x)] < \infty$. 
\end{proposition} 
\begin{proof} 
Existence of a solution is a straightforward consequence of the existence part in the statement of 
Proposition 
\ref{prop:existence:!}. This is Proposition 1.2 in 
\cite{lacker2022superposition}. Uniqueness is more difficult to obtain. We invoke Theorem 1.3 
in 
\cite{lacker2022superposition}. In brief, any solution $(\widetilde m_t)_{t \geq 0}$ 
(in the weak sense) 
to \eqref{eq:SFKP:sigma0} induces 
a weak solution to 
the McKean-Vlasov equation 
$\widetilde X=(\widetilde X_t)_{t \geq 0}$
to \eqref{LP}, weak in the sense that the private (or idiosyncratic) noise, say $\widetilde B$, becomes part of the 
solution. It holds $\widetilde m_t = {\mathcal L}(\widetilde X_t \vert {\mathcal F}_T^0)$. In fact, by strong uniqueness 
to \eqref{LP}, a relevant form of Yamada-Watanabe theorem applies to the current setting and implies that 
${\mathcal L}(\widetilde X_t \vert {\mathcal F}_T^0)$ is necessarily equal to 
${\mathcal L}_0(X_t)$, with $X=(X_t)_{t \geq 0}$ 
being the solution given by
Proposition  
\ref{prop:existence:!}. 
\end{proof} 
}


	\subsection{Restoration of uniqueness}
We now collect our main results about the long time behaviour of the process $(m_t)_{t \geq 0}$, solution of Equation \eqref{eq:SFKP:sigma0}. {Before we do so, we first state the precise definition of an invariant measure for such a stochastic process, when seen as a process with values in the space of probability measures}: 
\begin{definition}[Invariant measure]\label{D2}
We say that $\widebar P \in \mathcal{P}_{\textcolor{black}{2}}(\mathcal{P}_{{1}}(\R^d))$ {is} an invariant measure for the process $(m_t)_{t\geq 0}$ {when the latter is regarded as taking values in ${\mathcal P}_1({\mathbb R}^d)$}, if the law of $m_t$ is independent of $t$, when $m_0$ is distributed according to $\widebar P$, 
{i.e.}, for all continuous and bounded function $\phi \in \mathcal{C}_b(\mathcal{P}_{{1}}(\R^d))$
	\begin{align*}
		 {\E_0[\phi(m_t)]} = \E_0[\phi(m_0)] = \int_{\mathcal{P}_{{1}}(\R^d)} \phi(m) \widebar P(\mathrm{d}m), \quad \forall t > 0. 
	\end{align*}
\end{definition}
Of course, if $\bar P$ is an invariant probability distribution, then the process $(m_t)_{t\geq 0}$ with initial condition $\bar P$, has the same law as the process  $(m_{t+T})_{t\geq 0}$ for any $T > 0$.  This follows directly from the fact that $m_0$ and $m_T$ have the same law combined with the weak Markov property.
{Moreover, it is worth observing that, under the condition $\bar P \in \mathcal{P}_{\textcolor{black}{2}}(\mathcal{P}_{{1}}(\R^d))$ imposed in the statement, 
the function $m \mapsto \mathrm{d}_1^{{\mathbb R}^d}(\delta_0,m) = \int_{{\mathbb R}^d} \vert x \vert m({\mathrm d}x)$ is square integrable. 
This constraint may seem rather artificial. In fact, given the form of the dynamics in Equation \eqref{eq:SFKP:sigma0}, it is natural to 
require the invariant measures to be supported by ${\mathcal P}_1({\mathbb R}^d)$. Then, the aforementioned  integrability condition provides an additional growth property that makes easier the identification of those invariant measures, see 
the proof of Proposition 
\ref{prop:existence} right below.}

\vskip 4pt

Let us begin with the existence of an invariant measure $\bar P$:
\begin{proposition}\label{prop:existence}
	Under Assumptions~\ref{A1} and~\ref{A2}, there exists at least one invariant measure $\bar P$ for the process defined by Equation \eqref{eq:SFKP:sigma0}. {\color{black} Moreover,
	under the condition $\alpha > L_G$, any invariant measure belongs to $\mathcal{P}_2(\mathcal{P}_2(\R^d))$.} 
\end{proposition}

The proof of Proposition~\ref{prop:existence} \textcolor{black}{makes use of the notion of derivatives on} the space of measures. \textcolor{black}{Following \cite{cardaliaguet2010notes}
(see also \cite[Chap. 5]{CD2018_1} and \cite[Chap. 4]{CD2018_2})}, we recall the following definition:
\begin{definition}\label{D4}
Let us define $\mathcal{C}_b^2(\mathcal{P}_2(\mathbb{R}^d))$ as the collection of continuous and bounded functions $\Phi: \mathcal{P}_2(\mathbb{R}^d) \to \mathbb{R}$ with the following properties
\textcolor{black}{(throughout the definition, ${\mathcal P}_2({\mathbb R}^d)$ is equipped with $\mathrm{d}_2^{{\mathbb R}^d}$)}:
\begin{itemize}
\item There exists a unique 
\textcolor{black}{jointly continuous  function $\partial_m \Phi: (m,x) \in \mathcal{P}(\mathbb{R}^d) \times \mathbb{R}^d \to \partial_m \Phi(m,x) \in \mathbb{R}^d$, 
at most of quadratic growth in $x$ for any $m$,}
 such that 
\[
\lim _{h \to 0} \frac{\Phi(m+h(m^{\prime}-m))-\Phi(m)}{h}=\int_{\mathbb{R}^d} \partial_m \Phi(m, v)(m^{\prime}-m)(\mathrm{d} v),
\]
for all $m, m^{\prime} \in \mathcal{P}(\mathbb{R}^d)$ and
\[
\int_{\mathbb{R}^d} \partial_m \Phi(m, v) m(\mathrm{d} v)=0, \quad m \in \mathcal{P}(\mathbb{R}^d);
\]
\item 
\textcolor{black}{For any $m \in {\mathcal P}_2({\mathbb R}^d)$}, 
the mapping $x \mapsto \partial_m \Phi(m, x)$ is differentiable, \textcolor{black}{with the gradient being 
denoted 
$D_m \Phi(m, x)$; the mapping $(m,x) \mapsto D_{m} \Phi(m,x)$
is jointly continuous in $(m,x)$ and at most of linear growth in $x$ uniformly in $m$ in bounded subsets of ${\mathcal P}_2({\mathbb R}^d)$};
\item For any  $x \in \mathbb{R}^d$, every component of the $\mathbb{R}^d$-valued function $m \mapsto D_m \Phi\left(m, x\right)$ satisfies the same conditions as 
\textcolor{black}{in 
the first bullet point}, resulting in a mapping $(m,x,y) \in {\mathcal P}_2({\mathbb R}^d) \times {\mathbb R}^d \times {\mathbb R}^d \mapsto D_m^2 \Phi(m, x, y) \in \mathbb{R}^{d \times d}$, 
\textcolor{black}{which is jointly continuous (in the three arguments) and  at most of quadratic growth in $(x,y)$, uniformly in $m$ in bounded subsets of 
${\mathcal P}_2({\mathbb R}^d)$};
\item For any $m \in \mathcal{P}_2(\mathbb{R}^d)$, the 
\textcolor{black}{function $x \in {\mathbb R}^d \mapsto D_m \Phi(m,x)$ is differentiable; the Jacobian, denoted $(m,x) \mapsto D^2_{xm} \Phi(m, x)$, is jointly continuous in $(m, x)$
and at most of linear growth in $x$, uniformly in $m$ in bounded subsets of ${\mathcal P}_2({\mathbb R}^d)$}.
\end{itemize}
\end{definition}

\begin{proof}[Proof of Proposition~\ref{prop:existence}]

\textit{Step 1.}
	The proof of the existence of an invariant measure closely follows the approach taken in Proposition 2 of \cite{maillet2023note}, using the results of Section 1.2 of \cite{lacker2022superposition}. The key step therein is to control, uniformly in time, the second-order moments of the solutions to Equation \eqref{eq:SFKP:sigma0}, which can be achieved with relative ease in our context. Details are left to the reader. 

\vskip 4pt

	{\color{black}
	\textit{Step 2.}
We now address the second part of the statement (any invariant measure in 
	${\mathcal P}_2({\mathcal P}_1(\R^d))$ is in fact in 
	${\mathcal P}_2({\mathcal P}_2(\R^d))$). 
	Let us \textcolor{black}{thus} consider an invariant measure $\bar P \in \mathcal{P}_2(\mathcal{P}_1(\R^d))$
	and 
	call $m_0$ an ${\mathcal F}_0$-valued random variable with values in 
	${\mathcal P}_1(\R^d)$ such that $\mathcal{L}_0(m_0) = \bar P$. We denote by 
	$(m_t)_{t \geq 0}$   the solution to Equation \eqref{eq:SFKP:sigma0} with 
	$m_0$ as initial condition.
	We also introduce for  any $R > 0$,  a smooth 
	\textcolor{black}{function $\chi_R^d : \R^d \to \R^d$ that coincides with the identity on the ball $B_{{\mathbb R}^d}(0,R)$ and that is equal to $0$ outside 
	$B_{{\mathbb R}^d}(0,2R)$. We can assume the Lipschitz constant of $\chi_R^d$ to be bounded by $1$.} For any $R > 0$, \textcolor{black}{we denote by $P^R_0$ the law of $\chi_R^d\sharp m_0$ (where $\sharp$ is the pushforward operator) and construct 
	 $(m^R_t)_{t \geq 0}$, the solution to Equation \eqref{eq:SFKP:sigma0} with 
	 $\chi_R^d \sharp m_0$ as initial condition. Letting $(P^R_t)_{t \geq 0} := (\mathcal{L}_0(m^R_t))_{t \geq 0}$}, 
we deduce
from the truncation performed on the initial condition and from Equation \eqref{eq:SFKP:sigma0} that, for any $R > 0$ and any $t \geq 0$, $P^R_t(\mathcal{P}_2(\R^d)) = 1$. 	 We obtain from
	 the chain rule proven in 
	 \cite[Chap. 5]{CD2018_1} and \cite[Chap. 4]{CD2018_2} (see \cite{lacker2022superposition}
	 for the form that is retained below)  
	 that for any function $\Phi \in \mathcal{C}_b^2(\mathcal{P}_2(\mathbb{R}^d))$, for any $t \geq 0$,
	\begin{equation}\label{eq:cara}
		\langle P^R_t - P^R_0; \Phi\rangle = \int_0^t \langle P^R_s \textcolor{black}{;} \mathcal{M}\Phi\rangle\mathrm{d}s,
	\end{equation}
where for any $m \in \mathcal{P}_2(\R^d)$, 
	\begin{align*}
		\mathcal{M} \Phi(m):= & \int_{\mathbb{R}^{d}}\left[D_{m} \Phi(m, x) \cdot (G(x) + F(x-\mu_1(m)) +\frac{\sigma^2+ \sigma_0^2}{2}  \mathrm{Tr}\left[D^2_{xm} \Phi(m,x)\right]\right] m(\mathrm{d}x) \\
&+\frac{\sigma_0^2}{2} \int_{\mathbb{R}^{2d}} \mathrm{Tr}\left[D_{mm}^{2} \Phi\left(m, x, y\right)\right] m(\mathrm{d}x) m(\mathrm{d}y).
	\end{align*}

Let us now consider $\Phi(m) := 
\int_{{\mathbb R}^d} \phi(x-\mu_1(m)) m(\mathrm{d}x)$,
for a function $\phi : {\mathbb R}^d \to \R$ which we assume to be bounded and smooth, with bounded derivatives of order 1 and 2. In particular, 
$\Phi$ is bounded. Now, \textcolor{black}{$m\mapsto \Phi(m) \in\mathcal{C}^2_b(\mathcal{P}_2(\R^d))$ and we can write,} for any $m \in \mathcal{P}_2(\R^d)$ \textcolor{black}{and} $x,y \in \R^d$, 
\begin{align*}
& D_m\Phi(m,x)= \nabla \phi (x- \mu_1(m)) -  \int_{{\mathbb R}^d} \nabla \phi(y - \mu_1(m))m(\mathrm{d}y), 
\\
& D^2_{xm}\Phi(m,x)= \nabla^2 \phi (x- \mu_1(m)),
\\
& D^2_{mm}\Phi(m,x,y)= -  \nabla^2\phi (x- \mu_1(m))
- \nabla^2\phi (y- \mu_1(m))
+
 \int_{{\mathbb R}^d} \nabla^2\phi (z- \mu_1(m))m(\mathrm{d}z).
\end{align*}
Plugging this into Equation \eqref{eq:cara}, we obtain
\begin{equation*}
\begin{aligned}
	 \int_{\mathcal{P}_2(\R^d)} &\Phi(m)  (P^R_t - P^R_0)(\mathrm{d}m) \\
	&= \int_0^t \bigg[ \int_{\mathcal{P}_2(\R^d)}\int_{\R^d} \nabla\phi(x- \mu_1(m)) \cdot(G(x) + F(x-\mu_1(m)))m(\mathrm{d}x)P^R_s(\mathrm{d}m) 
	\\
	&-\int_{\mathcal{P}_2(\R^d)} \biggl( \int_{\R^d} \nabla\phi(y- \mu_1(m)) m(\mathrm{d}y) \biggr) \cdot \biggl( \int_{\R^d} (G(x) + F(x-\mu_1(m)))m(\mathrm{d}x) \biggr) \, P^R_s(\mathrm{d}m) 
	\\
	& + \frac12 \bigl( \sigma^2 + \sigma_0^2 \bigr) 
	 \int_{{\mathbb R}^d} {\rm Tr} \bigl[ \nabla^2\phi (z- \mu_1(m))\bigr] m(\mathrm{d}z)
 - \frac12 \sigma_0^2 	 \int_{{\mathbb R}^d} {\rm Tr} \bigl[ \nabla^2\phi (z- \mu_1(m))\bigr] m(\mathrm{d}z) \biggr] \, \mathrm{d}s,
\end{aligned}
\end{equation*}
which gives (by simplifying the last line) 
\begin{equation}\label{eq:aaa}
\begin{aligned}
	\int_{\mathcal{P}_2(\R^d)} &\Phi(m) (P^R_t - P^R_0)(\mathrm{d}m) \\
	& = \int_0^t \bigg[ \int_{\mathcal{P}_2(\R^d)}\int_{\R^d} \nabla\phi(x- \mu_1(m)) \cdot(G(x) + F(x-\mu_1(m)))m(\mathrm{d}x)P^R_s(\mathrm{d}m) 
	\\
	&-\int_{\mathcal{P}_2(\R^d)} \biggl( \int_{\R^d} \nabla\phi(y- \mu_1(m)) m(\mathrm{d}y) \biggr) \cdot \biggl( \int_{\R^d} (G(x) + F(x-\mu_1(m)))m(\mathrm{d}x) \biggr) \, P^R_s(\mathrm{d}m) 
	\\
	& + \frac12   \sigma^2  	 \int_{{\mathbb R}^d} {\rm Tr} \bigl[ \nabla^2\phi (z- \mu_1(m))\bigr] m(\mathrm{d}z) \biggr] \, \mathrm{d}s. 
\end{aligned}
\end{equation}

\textit{Step 3.} 
We now want to let $R \to \infty$. Let us begin with the term on the left-hand side
\begin{align*}
	\int_{\mathcal{P}(\R^d)} \Phi(m) (P^R_t - P^R_0)(\mathrm{d}m) = \int_{\mathcal{P}(\R^d)} \Phi(m) (P^R_t - \bar P)(\mathrm{d}m) + \int_{\mathcal{P}(\R^d)} \Phi(m) (\bar P - P^R_0)(\mathrm{d}m).
\end{align*}
\textcolor{black}{Using} the Lipschitz-continuity of $F$ and $G$, we can show that,
for any \textcolor{black}{$T \geq 0$}, there exists a constant \textcolor{black}{$C_T \geq 0$, such that, for any 
$t \in [0,T]$}, 
\begin{equation}\label{eq:jhfbjzd}
	\mathrm{d}_2^{\mathcal{P}_1(\R^d)}(P^R_t, \bar P) \leq \textcolor{black}{C_T} \: \mathrm{d}_2^{\mathcal{P}_1(\R^d)}(P^R_0, \bar P).
\end{equation}
Moreover, we can write
\begin{equation*}
	\mathrm{d}_2^{\mathcal{P}_1(\R^d)}(P^R_0, \bar P)^2 \leq \int_{\mathcal{P}(\R^d)} \biggl( \int_{\R^d} |\chi_R^d(x) - x| m(\mathrm{d}x) \biggr)^2 \bar P(\mathrm{d}m). 
\end{equation*}
Since $\chi_R^d$ is equal to $0$ in $0$ and is 1-Lipschitz continuous
and since 
$\bar P \in \mathcal{P}_2(\mathcal{P}_1(\R^d))$, the above right-hand side 
tends to $0$ as $R$ tends to $\infty$, that is 
\begin{equation}\label{eq:jvfebn}
\lim_{R \rightarrow + \infty} \mathrm{d}_2^{\mathcal{P}_1(\R^d)}(P^R_{\textcolor{black}{0}}, \bar P) = 0.
\end{equation}
Equations \eqref{eq:jhfbjzd} and \eqref{eq:jvfebn} say that for all $T \geq 0$, 
\begin{equation}
\lim_{R \rightarrow + \infty} \sup_{0 \le t \le T} \mathrm{d}_2^{\mathcal{P}_1(\R^d)}(P^R_{t}, \bar P) = 0.
\end{equation}
Back to \eqref{eq:aaa}, we observe that all the functions 
of $m$ that are integrated with respect to $\bar P$ therein are 
continuous in 
$m$ with respect to 
$\mathrm{d}_1^{{\mathbb R}^d}$. Moreover, they are all at most of linear growth with respect to $\int_{{\mathbb R}^d} \vert x \vert 
m({\mathrm d}x)$. This makes it possible to let $R$ tend to $+\infty$. We get 
\begin{equation*} 
\begin{split} 
&\int_{\mathcal{P}_1(\R^d)}\int_{\R^d} \nabla\phi(x- \mu_1(m)) \cdot(G(x) + F(x-\mu_1(m)))m(\mathrm{d}x) \bar P (\mathrm{d}m) 
	\\
	&-\int_{\mathcal{P}_1(\R^d)}   \biggl( \int_{\R^d} \nabla\phi(y- \mu_1(m)) m(\mathrm{d}y) \biggr) \cdot \biggl( \int_{\R^d} (G(x) + F(x-\mu_1(m)))m(\mathrm{d}x) \biggr) \, \bar P(\mathrm{d}m) 
	\\
	& + \frac12 
	 \sigma^2 
	 \int_{\mathcal{P}_1(\R^d)}
	 \int_{{\mathbb R}^d} {\rm Tr} \bigl[ \nabla^2\phi (z- \mu_1(m))\bigr] m(\mathrm{d}z)
  \bar P (\mathrm{d}m)  = 0.
  \end{split}
  \end{equation*} 
\vskip 4pt

\textit{Step 4.} Now, we take $\phi(x)$ in the form $\vartheta(\vert x \vert)$
where $\vartheta$ is a smooth function from $[0,+\infty)$ into itself, that is equal to
the identity 
on $[0,A]$ and that is equal to $3A/2$ on $[2A,+\infty)$, for some $A>0$. 
We assume 
$\vartheta'$ to be non-negative and bounded by the identity
and 
$\vartheta''$ to take values in $[0,2]$. 
Under this choice, $\nabla\phi(x)$ rewrites $[\vartheta'(\vert x \vert) / \vert x \vert] x$ and $\mathrm{Tr}[ \nabla^2\phi(x) ]$
is bounded by a constant $c_d$ only depending on the dimension $d$. Then, 
we can find a constant $C$, only depending on $G$ (and whose value is allowed to change from line to line) such that, for any $m \in \mathcal{P}_2(\R^d)$, 
\begin{align*}
&\biggl\vert  \int_{\R^d} \nabla\phi(x- \mu_1(m)) \cdot G(x)  m(\mathrm{d}x)  \biggr\vert
\\
&\leq 
\biggl\vert  \int_{\R^d} \nabla\phi(x- \mu_1(m)) \cdot \bigl[ G(x) - G(\mu_1(m))\bigr]  m(\mathrm{d}x)  \biggr\vert
+
\biggl\vert  \int_{\R^d} \nabla\phi(x- \mu_1(m)) \cdot G(\mu_1(m))  m(\mathrm{d}x)  \biggr\vert
\\
&\leq L_G \int_{\R^d} \vartheta'( \vert x - \mu_1(m) \vert ) \vert x - \mu_1(m) \vert  m(\mathrm{d}x)
	+ C 
	\int_{\R^d} \vartheta'( \vert x - \mu_1(m) \vert ) \bigl( 1 +\vert  \mu_1(m) \vert \bigr)  m(\mathrm{d}x)
\\
&\leq L_G 
\int_{\R^d} \vartheta'( \vert x - \mu_1(m) \vert ) \vert x - \mu_1(m) \vert  m(\mathrm{d}x)
+ C \bigl( 1 +  \mathrm{d}_1^{{\mathbb R^d}}(\delta_0,m)^2 \bigr). 
\end{align*}  
We insist on the fact that $C$ is independent of $\vartheta$. 
Now using the assumption on $F$ (see Assumption \ref{A2}), we obtain
\begin{align*}
	& \int_{\R^d}\nabla\phi(x- \mu_1(m))\cdot F(x-\mu_1(m))m(\mathrm{d}x)  \leq - \alpha  \int_{\R^d}\vartheta'( \vert x - \mu_1(m) \vert )
	\vert x - \mu_1(m)|
	m(\mathrm{d}x). 
\end{align*}
Rewriting the conclusion of the third step gives
\color{black}
\begin{equation}
\label{eq:lower:bound:inv:measure} 
\begin{split}
	&(\alpha - L_G) \int_{\mathcal{P}_1(\R^d)}\int_{\R^d} \vartheta'(\vert x- \mu_1(m)\vert) |x- \mu_1(m)|
	m(\mathrm{d}x)\bar P(\mathrm{d}m) 
	\\
	& \leq  {\color{black}c_d\sigma^2} + C \biggl( 1 +  \int_{\mathcal{P}_1(\R^d)} \mathrm{d}_1^{\R^d} (\delta_0,m)^2 \bar P( \mathrm{d}m) \biggr) . 
\end{split}
\end{equation}
Now, we can choose $\vartheta'$ along a non-decreasing sequence converging pointwise to the identity function on 
$[0,+\infty)$. By monotone convergence, we obtain 
\begin{equation*}
\begin{split}
	&(\alpha - L_G) \int_{\mathcal{P}_1(\R^d)}\int_{\R^d}  |x- \mu_1(m)|^2
	m(\mathrm{d}x)\bar P(\mathrm{d}m) 
	  \leq  {\color{black}{c_d\sigma^2}} + C \biggl( 1 +  \int_{\mathcal{P}_1(\R^d)} \mathrm{d}^{\R^d}_1(\delta_0,m)^2 \bar P( \mathrm{d}m) \biggr), 
\end{split}
\end{equation*}
which completes the proof. 
	}
\end{proof}
 
\vskip 4pt

Uniqueness is much more challenging to establish. When the intensity of the idiosyncratic noise $\sigma$ is positive, this question has not been addressed yet. The following proposition is one key step in this regard and demonstrates that, under certain assumptions, the initial condition is forgotten in long time. 

\begin{proposition}\label{prop:p1}
{\color{black}For any $\alpha_F > \max(L_G, m_G)$ in Assumption~\ref{A2}, for any $\sigma_0 > 0$, there exists $\widebar \sigma > 0$, depending on $F,G,\sigma_0$ and $d$, such that for all $\sigma \leq \widebar\sigma$},
{for any $1$-Lipschitz continuous  function $\phi : \mathcal{P}_1(\R^d) \rightarrow {\mathbb R}$ (with $\mathcal{P}_1(\R^d)$ being equipped with $\mathrm{d}_1^{\R^d}$)}
and any $m, \widetilde m \in \mathcal{P}_2(\R^d)$,
\begin{equation*}
	\left|\mathscr{P}_t\phi(m) - \mathscr{P}_t\phi(\widetilde m)\right| \leq \mathcal{K}(m, \widetilde m) e^{- \mathfrak{c}t},
	{\quad t \geq 0,}
\end{equation*}
for some $\mathfrak{c} := c(G, F, \sigma, \sigma_0) > 0$
{and some ${\mathcal K}(m,\widetilde m)$, 
which is described next}. 

{Indeed, given the} choice of $\alpha_F$, there 
{exist a} constant $c^{\alpha_F}_1 > 0$, {depending on $G$ and $F$}, {and a constant $C_{d}$, depending on $d$ (and thus independent 
of $\phi$, $m$ and $\widetilde m$), such that}
\[\mathcal{K}(m, \widetilde m) := C_{d}(\mu_2(m)^{1/2} + \mu_2(\widetilde m)^{1/2})\exp\left(\frac{c^{\alpha_F}_1}{{c_{\alpha_F}}}\left(v(m)^{1/2} + v(\widetilde m)^{1/2}\right) \right), \]
where ${c_{\alpha_F}}:= {\alpha_F}- L_G$.
{We} recall that $\mu_1$, $\mu_2$ and $v$ are defined in Section~\ref{sec:notation}. 
\end{proposition}

Proposition~\ref{prop:p1} shows that we forget the initial condition at an exponential rate for any sufficiently integrable initial conditions $m$ and $\widetilde m$. This is a key result of the paper, the proof of which is postponed to Section~\ref{proofP1}.

 As a consequence of Proposition~\ref{prop:p1}, {we get uniqueness of the invariant measure 
 and deduce that the system converges exponentially fast toward the stationary regime, both  facts being stated in the main statement below}. 
\begin{theorem}\label{th:main}
	{\color{black}For any $\alpha_F > \max(L_G, m_G)$ in Assumption~\ref{A2}, for any $\sigma_0 > 0$, there exists $\widebar \sigma > 0$, depending on $F,G,\sigma_0$ and $d$, such that for all $\sigma \leq \widebar\sigma$}, there is a unique invariant measure $\widebar P \in \mathcal{P}_2(\mathcal{P}_2(\R^d))$  {for the process $(m_t)_{t \geq 0}$, solution of Equation \eqref{eq:SFKP:sigma0}}.  More precisely, there exists {a constant $C > 0$}, such that for any probability measure $m \in \mathcal{P}(\R^d)$ and any 1-Lipschitz continuous function $\phi: \mathcal{P}_{{1}}(\R^d) \to \R$, (with $\mathcal{P}_1(\R^d)$ being equipped with $\mathrm{d}_1^{\R^d}$), we have 
	\begin{equation*}
		\left| \mathscr{P}_t\phi(m) - \int_{\mathcal{P}(\R^d)}\mathscr{P}_t\phi(\widetilde m) \widebar P(\mathrm{d} \widetilde m) \right| \leq C \bigl( 1+ \mu_2(m)^{1/2} \bigr)\exp(\lambda v(m)^{1/2})e^{-\mathfrak{c} t}, {\quad t \geq 0,}
	\end{equation*}
	with {$\lambda:= {c_1^{\alpha_F}}/{c_{\alpha_F}}>0$, and for 
	$\mathfrak{c}$, $c_{\alpha_F}$ and $c_1^{\alpha_F}$  as
	in the statement of Proposition~\ref{prop:p1}.} 
\end{theorem}

{\color{black}
\begin{remark}
\label{rem:choice:parameters}
Let us comment here on the \textcolor{black}{choice of parameters in the above statement}. First, it is important to note that
\textcolor{black}{the type and strength of the 
 interaction play a central role in our framework. The interaction, in both its form and intensity ($\alpha_F \geq \widebar{\alpha}$), forces 
the long-time behaviour of the entire process $(m_t)_{t \geq 0}$ to be  dictated by its mean $(\mu_1(m_t))_{t \geq 0}$, which is the cornerstone of the proof}. \textcolor{black}{Asking the interaction to be strong, as we require in the condition $\alpha_F \geq \widebar{\alpha}$, may seem surprising at first sight}. Indeed, \textcolor{black}{this is fundamentally different from the standard picture that exists for models without common noise, where strong interactions are typically expected to give rise to several invariant measures}. 
\textcolor{black}{In our case, the presence of the common noise, even of a possibly small intensity, induces a phase transition. 
For any positive value of $\sigma_0$ and for a sufficiently large interaction (independently of the value of $\sigma_0$), we can indeed find $\sigma >0$ such that uniqueness of the invariant measure is ensured in presence of a common noise of intensity 
$\sigma_0$ but is lost when the common noise is disabled. This makes a conceptual difference with the previous results obtained on the subject in  
\cite{jianhai2024long}, where 
uniqueness is restored when the global contribution of the two noises is large enough with respect to the strength of the interaction. In comparison, $\sigma$ and $\sigma_0$ 
have opposite roles to each other in our approach.}
%
%
%
%
\end{remark}
}

	\subsection{Discussions and examples}\label{sec:motivation} In this section, we present two explicit examples in which 
	{common noise forces} uniqueness of the invariant measure. The first example fits 
	{explicitly} the framework of this work {and the second one may be seen as
	a variant of 
	\eqref{eq:SFKP:sigma0}. At last, we also provide a counter-example where uniqueness does not hold despite the presence of the common noise.}

\subsubsection*{A prototype: dynamics driven by a confining potential and a linear interaction.}
{Let us consider the
 case $G(x)=-\nabla V(x)$, $x \in {\mathbb R}^d$, for a differentiable function $V : {\mathbb R}^d \rightarrow {\mathbb R}$,
 and $F(x) = - \alpha x$, $x \in {\mathbb R}^d$, for a certain $\alpha >0$. Then, 
 \eqref{eq:SFKP:sigma0} becomes}
 \begin{equation}\label{eq:example1}
	\mathrm{d}_tm_t = \frac{{\sigma^2 + } \sigma_0^2}{2}\Delta m_t\mathrm{d}t + \nabla\cdot\left[m_t\Bigl(\nabla V + \alpha\left(\cdot - \mu_1(m_t) \right) \Bigr)\right] \mathrm{d}t- \sigma_0\nabla m_t\cdot\mathrm{d}B^0_t,
	\quad {t \geq 0}. 
\end{equation}
{When $\sigma_0 = 0$, 
this equation reduces to} 
\begin{equation}
\label{eq:example1:sigma=0}
	\partial_t m_t = \frac{{\sigma^2}}{2}\Delta m_t + \nabla\cdot\left[m_t\Bigl(\nabla V + \alpha\left(\cdot - \mu_1(m_t) \right) \Bigr)\right]. 
\end{equation}
When $V$ is a double well potential, the {latter} equation 
{may have several stationary solutions}. For instance, 
{if $\sigma=0$ (in addition to the assumption $\sigma_0=0$), 
any critical point $x_0$ of $V$ induces a stationary solution, concentrated 
at $x_0$. In the standard example where 
$d=1$ and $V(x) = |x|^4/4 - |x|^2/2$ 
(see for instance \cite{1983JSP....31...29D,herrmann2010non}),  
those stationary solutions are the Dirac masses $\delta_1$, $\delta_0$ 
and $\delta_{-1}$. 
Non-uniqueness persists when $\sigma$ is strictly positive but small:
with the same potential $V$ and for $\sigma$ less than a certain threshold 
$\bar \sigma$, there are three invariant measures, respectively centered around 
$1$, $0$ and $-1$, 
to the McKean-Vlasov SDE
\begin{align*}
	\mathrm{d}X_t = - V'(X_t)\mathrm{d}t - \alpha(X_t - \E_1[X_t]) \mathrm{d}t + \sigma \mathrm{d}B_t, \quad t \geq 0.
\end{align*}
Lack of uniqueness can be explained as follows. There is a competition between 
the noise $B$ (weighted by the 
intensity factor $\sigma$) and the interaction term function $(x,m) \mapsto 
-\alpha (x - \mu_1(m))$. On the one hand, the process $X$ is attracted by the minimizers of $V$ and by its expectation. 
On the other hand, it is subjected to the diffusive effect of the noise. Whenever the coefficient $\sigma$ is too small, the 
interaction term dominates. Any invariant measure has a small variance and must be concentrated around a minimizer of 
$V$, hence forcing the multiplicity of the stationary solutions to \eqref{eq:example1:sigma=0}.}


{When $\sigma_0>0$ but $\sigma=0$ (and under conditions similar to 
Assumption~\ref{A1}), whenever $\alpha$ is assumed large enough, the second author}
has shown in 
 \cite{maillet2023note} 
that the process $(m_t)_{t \geq 0}$ {admits} a unique invariant measure $P_0 \in \mathcal{P}(\mathcal{P}(\R^d))$,
{and thus established restoration of uniqueness in this situation. In this particular case,} the invariant measure is supported by Dirac masses:
\begin{equation*}
	\widebar{P}_0(\mathrm{d}m) = \int_{\R^d} \delta_{\delta_a}{(\mathrm{d}m)} m^*(\mathrm{d}a),
\end{equation*}
where $m^*$ is solution of the stationary Fokker-Planck equation $(\sigma_0^2/2) \Delta m^* {
+ \nabla\cdot(m^*\nabla V)=0}$. 
{Moreover, as consequence of the contracting properties of $F$, it can be shown} that 
\begin{equation*}
	v(m_t) \to 0,
\end{equation*}
when $t \to +\infty$, {which implies that, asymptotically (in time), the solution $X$ 
to 
\eqref{LP}
coincides with its conditional expectation given the common noise, hence justifying the shape of the measure $\widebar{P}_0$}. 

{In the case $\sigma_0>0$ and $\sigma$ small but (strictly) positive, our result extends 
the analysis carried out in \cite{maillet2023note}. However, the long-time behaviour of $(v(m_t))_{t \geq 0}$ is different 
and the latter does not tend to $0$ as $t$ tends to $\infty$. This difference is substantial and makes the proof significantly more challenging in our analysis. The argument is explained in the next two sections. In short, one shows that the 
long-time behavior of the
conditional mean 
$(\mu_1(m_t))_{t \geq 0}$ coincides with that one of the standard SDE
$\mathrm{d}Y_t = -\nabla V(Y_t) \mathrm{d} t +  \sigma_0\mathrm{d}B^0_t$,
	${t \geq 0}$ up to an error that is small with $\sigma$ 
	and $1/\alpha$. Although the fluctuation 
	$(X_t - {\mathbb E}_1(X_t) = X_t - \mu_1(m_t))_{t \geq 0}$ does not vanish in long time, which is 
	the main mathematical difficulty here, 
	the regime $0 < \sigma \ll 1$ shares many conceptual similarities with the case $\sigma=0$. In particular, the long-time behavior of 
	$(m_t)_{t \geq 0}$ is mostly dictated by the ergodic properties of the standard SDE obtained by putting 
$F \equiv 0$ and $\sigma=0$ in 
\eqref{LP}.  This example is typical of our study.}
When the value of $\sigma$ is significantly large and $\alpha$ is relatively small, it is straightforward to establish the uniqueness of the invariant measure using coupling arguments on the idiosyncratic noise, as in \cite{durmus2020elementary}. Still, it is fair to say that we do not have, at this stage, a good understanding of the behaviour of the model between these two regimes.

\subsubsection*{A variant: dynamics driven by a quadratic potential and a nonlinear first order mean field term}
{In this paragraph, we cook up a variant 
of the model covered by 
Theorem~\ref{th:main} that does not exactly fit 
the form postulated in 
\eqref{eq:SFKP:sigma0} but that can be studied 
in a similar manner. 
This example is built up as a perturbation of an Ornstein-Uhlenbeck process. Precisely, 
we consider the stochastic Fokker-Planck equation} 
\begin{equation}\label{eq:ex:2}
	\mathrm{d}_tm_t = \frac{\sigma^2 + \sigma_0^2}{2}\Delta m_t\mathrm{d}t + \nabla \cdot 
	\left[m_t\bigl(ax - f(\mu_1(m_t)\bigr) \right]\mathrm{d}t - \sigma_0\nabla m_t\cdot\mathrm{d}B^0_t, {\quad t \geq 0}, 
\end{equation}
{whose} probabilistic counterpart {writes}
\begin{equation*}
	\mathrm{d}X_t = -a X_t \mathrm{d}t + f(\E_1[X_t])\mathrm{d}t + \sigma\mathrm{d}B_t + \sigma_0\mathrm{d}B^0_t,
	\quad {t \geq 0}, 
\end{equation*}
for some $a > 0$ and $f : \R^d \to \R^d$. {Equivalently,} the conditional law 
{$(\mathcal{L}_1(X_t))_{t \ge 0}$} of $X$ given 
the common noise $(\mathcal{L}_1(X_t))_t$ is a solution to \eqref{eq:ex:2}. 

{Obviously, the mean field term $f(\E_1[X_t])$ 
cannot be put in the form $F(X_t - \E_1[X_t])$, from which we see that this example is outside the scope of 
Theorem~\ref{th:main}. That said, it is in fact not that far from 
the framework addressed in this article 
and, in particular, it obeys phenomena similar  to those underpinning the proof of 
Theorem~\ref{th:main}.}

{Actually, 
the long-time analysis of equation 
\eqref{eq:ex:2} is rather straightforward}. 
When $\sigma_0 = 0$,  any invariant measure $m$ is solution to 
\begin{equation*}
	\frac{\sigma^2}{2}\Delta m - \nabla \cdot\left[m\bigl(f(\mu_1(m)) - ax\bigr)\right] = 0, \quad x \in {\mathbb R}^d. 
\end{equation*}
{By integrating with respect to $x$}, the previous equation {leads to the condition}
\begin{equation}\label{CN}
	f(\mu_1(m)) = a\mu_1(m),
\end{equation}
{which is in fact sufficient in the following sense: once $\mu_1(m)$ has been found, 
the entire measure $m$ can be defined as the invariant measure of the 
Ornstein-Uhlenbeck dynamics 
$\mathrm{d}X_t = -a (X_t-\mu_1(m)) \mathrm{d}t + \sigma \mathrm{d}B_t$,
$t \geq 0$. Therefore, Equation \eqref{CN} says that} there exist several invariant measures as soon as $f/a$ admits several fixed points in $\R^d$. 

{In presence of a} common noise (i.e., $\sigma_0 > 0$), {the conditional mean 
$(\mu_1(m_t))_{t \geq 0}$ solves the SDE}
\begin{equation}
\label{eq:SDE:mu1mt:variant}
	\mathrm{d}\mu_1(m_t) = -a\mu_1(m_t)\mathrm{d}t + f(\mu_1(m_t))\mathrm{d}t + \sigma_0\mathrm{d}B^0_t, 
	\quad {t \geq 0},
\end{equation}
{and under appropriate confining conditions on 
the \textit{effective} drift $x \mapsto -a x + f(x)$ (which may be compatible with the fact that 
\eqref{CN} has several fixed points), 
\eqref{eq:SDE:mu1mt:variant}
admits a unique invariant measure, say $m^*$ (on $\R^d$), which implies that any two invariant measures on 
${\mathcal P}_1(\R^d)$ (in the sense of Definition~\ref{D2}) must have the same marginal law 
by the projection $m \mapsto \mu_1(m)$. As above, the entire invariant measure (on 
${\mathcal P}_1(\R^d)$) can be   recovered   by 
observing that 
$\mathrm{d}(X_t- {\mathbb E}_1(X_t)) = -a (X_t - {\mathbb E}_1(X_t)) \mathrm{d}t + \sigma \mathrm{d}B_t$,
 $t \geq 0$. This shows that the invariant measure is 
 \begin{equation}
 \label{eq:widebarP:variant} 
 \widebar{P}(\mathrm{d}m) = \int_{\R^d} \delta_{{\mathcal N}_d(\theta,[\sigma^2/(2a)] I_d)}{(\mathrm{d}m)} m^*(\mathrm{d} \theta),
\end{equation} 
where ${\mathcal N}_d(\theta,[\sigma^2/(2a)] I_d)$ is the $d$-dimensional Gaussian law
 with $\theta$ as mean and $\sigma^2/(2a) I_d$ as covariance 
($I_d$ standing for the $d$-dimensional identity matrix). In particular, there is a unique invariant measure even though 
\eqref{CN} has several fixed points.}

{The spirit of this new example is clear: similar to the prototype addressed in 
the previous paragraph, its long-time analysis is in fact governed by the 
simpler Fokker-Planck equation 
\eqref{eq:SDE:mu1mt:variant}
for the sole conditional mean. Here the situation is even simpler because the 
residual fluctuation 
$(X_t- {\mathbb E}_1(X_t))_{t \geq 0}$
has a trivial behavior (whilst the analysis of the residual is non-trivial in the prototype example).}

\subsubsection*{Combining the two examples}
{For sure, the reader may wonder about a global framework that would 
cover both the prototype and the variant examples. While this would be indeed possible to extend in such a way the current setting,  we have decided not 
to go up to this level of generality in order to keep the presentation and the notation at a reasonable level. In order to guess what  the more general form of 
\eqref{eq:SFKP:sigma0} should be, one first needs to understand which of the two coefficients $F$ and $G$ 
(in 
\eqref{eq:SFKP:sigma0}) 
correspond respectively to 
$x \mapsto - a x$ and 
$m \mapsto f(\mu_1(m))$  in 
\eqref{eq:ex:2}. At first sight, one may be tempted to regard 
$x \mapsto -a x$ as a specific example of $G$, but this is the wrong choice. 
In fact, the correct answer is to write $-ax$ as $-a(x- \mu_1(m))$, which prompts us to associate the 
coefficient $a$ in this model with the coefficient $\alpha$ in the prototype example. Next, 
we should see the remaining coefficient $m \mapsto -a \mu_1(m) + f(\mu_1(m))$ 
as a \textit{new} $G$, with the subtlety that, in our global framework, $G$ must be allowed to depend on 
both $x$ and $\mu_1(m)$. In clear, the new version of \eqref{eq:SFKP:sigma0} should be 
\begin{equation}\label{eq:SFKP:sigma0:ext}
	\mathrm{d}_t  m_t = \nabla \cdot \left(\frac{\sigma^2 + \sigma_0^2}{2} \nabla m_t - m_t \left[ G\bigl(\mu_1(m_t),\cdot\bigr) + F(\cdot - \mu_1(m_t))\right]\right)\mathrm{d}t -\sigma_0 \nabla m_t \cdot \mathrm{d}B^0_t, 
\end{equation}
for a function $G : {\mathbb R}^d \times {\mathbb R}^d \rightarrow {\mathbb R}^d$. Back to the prototype example 
addressed in the previous paragraph, this says that the long-time behavior
of $(m_t)_{t \geq 0}$ should now be compared with those of the standard SDE 
$\mathrm{d}Y_t = G(Y_t,Y_t) \mathrm{d} t +  \sigma_0\mathrm{d}B^0_t$, up to some fluctuation terms that 
one may expect to control properly if $F$ is sufficiently decreasing. In this context, what 
truly 
matters are the confining properties of the \textit{doubled} mapping $(x,x) \mapsto G(x,x)$.}

{We strongly believe that Theorem~\ref{th:main} could be extended to this setting, with 
a similar analysis. From a technical point of view, the gain would be limited: the difficulty would not come 
from the variable $\mu_1(m)$ in $G$ but from the variable $x$ (precisely because the strategy is to replace in the end 
the argument $X_t$ by its conditional expectation ${\mathbb E}_1(X_t)$ in \eqref{LP}). This difficulty is already present in 
the prototype example \eqref{eq:example1} and, in contrast, it is clear from the example
\eqref{eq:ex:2} that the presence 
of the variable $\mu_1(m)$ in $G$ does not raise any substantial difficulty. 
This explains how choice to work on \eqref{eq:SFKP:sigma0} (and not on \eqref{eq:SFKP:sigma0:ext}).}

%
\subsubsection*{An example without uniqueness recovery} As highlighted earlier, the finite-dimensional nature of the common noise constitutes a significant limitation {to obtain ergodic properties 
on the process $(m_t)_{t \geq 0}$ 
(see in contrast the recent work 
\cite{delarue2024rearranged} by the first author for a 1d case with an infinite dimensional noise).
In particular, this is the thrust of our work to identify one class of mean field dynamics 
for which the common noise really helps in this matter. However, it is clear that there is no chance to get uniqueness of the invariant measure for a generic class of 
stochastic Fokker-Planck SPDEs forced by a \textit{simple} noise like $B^0$.} 

A counter-example in dimension $d=1$ is 
\begin{equation*}
	\mathrm{d}_t m_t = \frac{\sigma(v(m_t))^2 + \sigma_0^2}{2} \partial_{xx}^2 m_t\mathrm{d}t + \partial_x{( x m_t)} \mathrm{d}t  - \sigma_0 \partial_x m_t\mathrm{d}B^0_t,
\end{equation*}
for some function $\sigma : [0, +\infty) \to (0, +\infty)$. Whenever $\sigma_0 = 0$ and 
{the function} $\sigma^2(\cdot)/4$ admits several fixed points, 
{the stationary solutions to the above equation  cannot be unique. In this case,} $(m_t)_{t\geq 0}$ is indeed solution of 
\begin{equation}\label{eq:sig0:example3}
	\partial_t m_t  = \frac{\sigma(v(m_t))^2}{2}\partial_{xx}^2\left( m_t\right) + \partial_x(x m_t), \quad t \geq 0. 
\end{equation}
{Then, similar to the stationary solutions found
in the previous example, 
the stationary solutions (equivalently, the invariant measures of 
\eqref{LP}) are here solutions of}
\begin{equation*}
{ \frac{\mathrm{d} m}{\mathrm{d} x}(x)} = \exp\left(-2\frac{x^2}{\sigma^2(v(m))}\right), \quad x \in {\mathbb R}. 
\end{equation*}
Then, one can show that, {if the function $x \mapsto \sigma^2(x)/4$ admits several fixed points,} then 
there are  several stationary solutions to \eqref{eq:sig0:example3}. 

When adding common noise, the situation does not get better. 
{For instance, assuming that the initial condition $P_0 =\delta_{\delta_0}$,} one can show that for all $t > 0$, $\mathbb{P}_0$-almost surely, $m_t$ is a {(random)} Gaussian probability measure with parameters
\begin{equation*}
\left\{
\begin{aligned}
	& \mu_1(m_t) = \sigma_0\int_0^t e^{s-t} \mathrm{d}B^0_s,\\
	& v(m_t) = \int_0^t e^{2(s-t)}\sigma^2(v(m_s)) \mathrm{d}s.
\end{aligned}
\right.
\end{equation*}
Then, it is not difficult to see that $(\mu_1(m_t))_{t\geq0}$ is an {ergodic process
(it is in fact an Ornstein-Uhlenbeck process with a negative mean-reverting parameter). As for 
the dynamics of the variance $(v(m_t))_{t\geq0}$}, we get
\begin{equation*}
	\partial_t v(m_t) = -2v(m_t) + \frac{1}{2}\sigma^2(v(m_t)), \quad t \geq 0.
\end{equation*}
{Obviously, the above equation has several stationary solutions (say $\sigma^*$) if 
the function
$x \mapsto \sigma^2(x)/4$ admits several fixed points. 
In the latter case, 
$(m_t)_{t \geq 0}$ has several invariant measures, whose form 
is similar to 
\eqref{eq:widebarP:variant} except that $\sigma$ therein is now replaced by any $\sigma^*$.}
In this specific example, it is noteworthy that {the common noise has no influence on the mean field term 
$(v(m_t))_{t \geq 0}$ (which is completely deterministic). Clearly, this is a consequence of the additive structure 
of the noise, the effect of which is just to shift (or to translate) the measures 
$(m_t)_{t \geq 0}$.}

\subsubsection*{Conclusion} {All these examples illustrate that
the form of the mean field interaction in 
\eqref{eq:SFKP:sigma0}, based on the sole (conditional) mean state 
is key in the derivation of 
Theorem~\ref{th:main}.
In particular, this structure makes it possible to transmit the noise from the dynamics of the 
state variable $(X_t)_{t \geq 0}$ in \eqref{LP} to its conditional 
mean $({\mathbb E}_1(X_t))_{t \geq 0}$. As suggested in the discussion on the prototype example, the next step is 
to prove that the long run behavior of 
$({\mathbb E}_1(X_t))_{t \geq 0}$ is dictated by the ergodic properties of the SDE 
$\mathrm{d}Y_t = G(Y_t) \mathrm{d} t +  \sigma_0\mathrm{d}B^0_t$, 
at least if 
$\alpha$ is large and $\sigma$ is small. This is the main line of the proof that is presented in the next two sections.}


\section{Proof of Theorem~\ref{th:main}}\label{sec:th1}

{In this section, we explain how to  deduce Theorem~\ref{th:main}
from Proposition~\ref{prop:p1}, but the proof of the latter  is postponed to Section~\ref{proofP1}.} 

{Before moving on to the proof of Theorem~\ref{th:main}, we need two preliminary estimates: the first one addresses the conditional variance of the solution to 
\eqref{eq:SFKP:sigma0}, and the second one the conditional exponential moments of the solution.}

\subsection{First preliminary estimate: conditional variance}\label{subsec:31}
{Given a solution 
$X$ to Equation \eqref{LP}, we first focus on} the conditional expectation given the common noise $(\E_1[X_t])_{t\geq 0}$:

\begin{proposition}\label{prop:approx}
	Let us consider a $d$-dimensional Brownian motion $\beta^0$ defined on $(\Omega_0, \mathcal{F}^0, \mathbb{P}_0)$ and adapted to the filtration $\mathbb{F}^0$. 
{\color{black} Let Assumptions~\ref{A1} and~\ref{A2} be in force with the additional constraint ${c_{\alpha_F}}:= {\alpha_F} - L_G > 0$. Let us consider $P_0 \in \mathcal{P}(\mathcal{P}_1(\R^d))$ and $X_0$ such that $\mathcal{L}_1(X_0) = m_0$ a.s, for some random measure $m_0$ satisfying $\mathcal{L}_0(m_0) = P_0$. 
	Then, for $X$ with dynamic given by}
	\begin{equation*}
		\left\{ \begin{array}{ll}
			\mathrm{d}X_t = G(X_t) \mathrm{d}t + F(X_t-\E_1[X_t])\mathrm{d}t + \sigma \mathrm{d}B_t + \sigma_0\mathrm{d}\beta^0_t\\
			X_{t|t=0} = X_0,
		\end{array}\right. 
	\end{equation*}
we have for all $t \geq 0$, 
	\begin{equation}\label{eq:approx}
		\mathbb{E}_1\left[|X_t - \E_1[X_t]|^2\right] \leq \E_1\left[|X_0 - \E_1[X_0]|^2\right]e^{-2{c_{\alpha_F}} t} + \frac{d\sigma^2}{{c_{\alpha_F}}},
	\end{equation} 
	where the previous inequality holds $\mathbb{P}_0-a.s$. 
\end{proposition}

	\begin{remark}		

$\bs{(1)}$ Inequality \eqref{eq:approx} does not depend on the choice of the Brownian motion $(\beta^0_t)_{t \geq 0}$ in the dynamics of $X$. 
{Accordingly, it can be recast in terms of the solution 
\eqref{eq:SFKP:sigma0} in the form 
\begin{equation}
\label{eq:v:mt}
	v(m_t) \leq v(m_0)e^{-2 {c_{\alpha_F}}t} + \frac{d\sigma^2}{{{c_{\alpha_F}}}}, \quad t \geq 0. 
\end{equation}}
\vspace{-10pt}

$\bs{(2)}$ {The result can be interpreted as follows}. When $\sigma=0$, the result is consistent with \cite{maillet2023note} and says that the process  $(X_t)_{t\geq 0}$ is attracted by its conditional expectation in the long run. When $\sigma > 0$, the result provides a sharp estimate of the residual {conditional} variance. 

$\bs{(3)}$ {Back to 
Theorem 
\ref{th:main}}, 
the challenge is {precisely} to prove that the long-time behaviour of $X$ is dictated by the long-time behaviour of $\E_1[X]$ even though the residual term is not zero. 
	\end{remark}

\begin{proof}[Proof of Proposition~\ref{prop:approx}]
	Using the dynamics of the process $X$, one can write
	\begin{equation*}
		\mathrm{d}(X_t-\E_1[X_t]) = (G(X_t) - \E_1[G(X_t)])\mathrm{d}t  + (F(X_t-\E_1[X_t]) - \E_1[ F(X_t - \E_1[X_t])])\mathrm{d}t + \sigma\mathrm{d}B_t
	\end{equation*}
	Then, applying Itô's formula, it comes
	\begin{align*}
		\mathrm{d}|X_t-\E_1[X_t]|^2 =  & \: 2(X_t - \E_1[X_t])\cdot(G(X_t) - \E_1[G(X_t)])\mathrm{d}t\\
		 +& \: 2(X_t - \E_1[X_t])\cdot(F(X_t-\E_1[X_t]) - \E_1[F(X_t- \E_1[X_t])])\mathrm{d}t \\
		 + & \: 2\sigma (X_t - \E_1[X_t])\cdot\mathrm{d}B_t \\
		 + & \: d \sigma^2 \mathrm{d}t
	\end{align*}
	Using the fact that $G$ is Lipschitz continuous and taking expectation, we get that
	\begin{align*}
		\frac{\mathrm{d}}{\mathrm{d}t}\E_1\left[|X_t-\E_1[X_t]|^2\right] \leq 2L_G\mathbb{E}_1\left[|X_t - \E_1[X_t]|^2\right] + 2\E_1[(X_t -\E_1[X_t])\cdot F(X_t - \E_1[X_t])] + d\sigma^2. 
	\end{align*}
	Moreover, we can write 
	\begin{align*}
		\E_1\left[(X_t -\E_1[X_t])\cdot F(X_t - \E_1[X_t])\right] & = \E_1\left[(X_t -\E_1[X_t])\cdot \left(F(X_t - \E_1[X_t]) - F(0)\right)\right] \\
		& \leq -{\alpha_F}\: \E_1\left[|X_t - \E_1[X_t]|^2\right],
	\end{align*}
	where we recall that ${\alpha_F}> 0$ is defined in Assumption~\ref{A2}.  Then, 
	\begin{equation*}
		\frac{\mathrm{d}}{\mathrm{d}t}\E_1[|X_t-\E_1[X_t]|^2] \leq -2({\alpha_F}- L_G)\mathbb{E}_1[|X_t - \E_1[X_t]|^2] + d\sigma^2. 
	\end{equation*}
	Finally, we get the result using Grönwall's Lemma, as soon as ${\alpha_F}> L_G$. 
\end{proof}
\subsection{{Second preliminary estimate: conditional exponential moments}}

Let us state the following lemma which allows us to control the exponential moments of the process $(m_t)_{t\geq 0}$ under {one} invariant measure $\widebar P$ {(which will be in the end `the' invariant measure)}:
\begin{lemma}\label{lemma:estimate}
Under Assumptions \ref{A1} and \ref{A2}, assuming that $\alpha_F > L_G$, 
{there exists an invariant measure $\widebar{P} \in \mathcal{P}_2(\mathcal{P}_2(\R^d))$  for the process $(m_t)_{t \geq 0}$,
such that 
\begin{equation}\label{eq:finite2}
	\int_{\mathcal{P}(\R^d)} \exp\left( 2c_2 v(\widetilde m)^{1/2}\right) \widebar P(\mathrm{d}\widetilde m)< +\infty, 
\end{equation}
	where $c_2 := {c_1^{\alpha_F}}/{c_{\alpha_F}}$.}
\end{lemma}

\begin{remark}
	The result presented in Lemma~\ref{lemma:estimate} 
	{aligns with the global picture we gave for} the behavior of the model. {In particular, it is natural (and in fact well-expected) in the regime $\sigma \ll 1$ (which 
	corresponds to the framework of Proposition~\ref{prop:p1}). Indeed, we already know from \cite{maillet2023note} (see also Section~\ref{sec:motivation}) that, if $\sigma = 0$, the invariant measure $\widebar{P}^0$ is supported by Dirac masses; therefore, the conclusion 
	of Lemma~\ref{lemma:estimate} becomes straightforward when replacing
	 $\widebar P$ by $\widebar P^0$ in the statement. The thrust of Lemma~\ref{lemma:estimate} is thus to extend the result to positive values of $\sigma$. Although our proof relies on a direct computation, it is worth noticing that  any invariant measure $\widebar P$ (for $\sigma >0$) should satisfy
	\begin{equation*}
		\mathrm{d}_1^{\mathcal{P}(\R^d)}(\widebar P, \widebar{P}^0) \leq C \sigma^2,
	\end{equation*}
for some constant $C$ independent of $\sigma$. In particular, any invariant $\widebar P$ is close to $\widebar P^0$ when $\sigma$ is small, hence justifying our intuition that Lemma~\ref{lemma:estimate} is a perturbation of the regime $\sigma=0$.}
\end{remark}

\begin{proof}[Proof of Lemma~\ref{lemma:estimate}]

Consider a probability measure $P_0 \in \mathcal{P}(\mathcal{P}(\R^d))$ 
{with $P_0({\mathcal P}_2({\mathbb R}^d))=1$}
 and the measure valued stochastic process $(m_t)_{t\geq 0}$ 
{solving 
\eqref{eq:SFKP:sigma0} with $\mathcal{L}_0(m_0) = P_0$ as initial condition}. 
Then, letting $P_t = \mathcal{L}_0(m_t)$, $\forall t \geq 0$, we can easily show with arguments similar to those {used} in the proof of Proposition 2 in \cite{maillet2023note} that the sequence $(\mathcal{Q}_T)_{T\geq 0}$ defined by 
\begin{equation*}
	\mathcal{Q}_T := \frac{1}{T} \int_0^T P_t \mathrm{d}t
\end{equation*}
admits a converging subsequence, still denoted by $({\mathcal{Q}}_T)_{T\geq 0}$,  that converges to an invariant measure $\widebar P$.  
\trash{Then, let us consider a probability measure $P_0 \in \mathcal{P}(\mathcal{P}(\R^d))$, such that
\begin{equation}\label{eq:integrability}
	\max\left(\int_{\mathcal{P}(\R^d)} \mu_2(\widetilde m) P_0(\mathrm{d}m) \: ; \: \int_{\mathcal{P}(\R^d)} \exp\left( 2c_2v(\widetilde m)^{1/2}\right)  P_0(\mathrm{d}\widetilde m) \right) < +\infty.
\end{equation}}
For any $T > 0$,
\begin{align*}
	\int_{\mathcal{P}(\R^d)} \exp\left( 2c_2 v(\widetilde m)^{1/2}\right)  \mathcal{Q}_T(\mathrm{d}\widetilde m) &=\frac{1}{T} \int_0^T \int_{\mathcal{P}(\R^d)} \exp\left( 2c_2 v(\widetilde m)^{1/2}\right)  P_t(\mathrm{d}\widetilde m) \\
	&  = \frac{1}{T}\int_0^T \E_0\left[\exp\left( 2c_2 v(m_t)^{1/2}\right)\right] \mathrm{d}t.
\end{align*}
Moreover,  {\eqref{eq:v:mt} gives} 
\begin{align*}
	\int_{\mathcal{P}(\R^d)} \exp\left( 2c_2 v(\widetilde m)^{1/2}\right)  \mathcal{Q}_T(\mathrm{d}\widetilde m) 
	& \leq \frac{1}{T}\int_0^T \E_0\left[\exp\left( 2c_2 v(m_0)^{1/2} + \frac{2c_2 \sigma}{\sqrt{{c_{\alpha_F}}}}\right)\right] \mathrm{d}t 
	\\
	& \leq \E_0\left[\exp\left( 2c_2 v(m_0)^{1/2} + \frac{2c_2 \sigma}{\sqrt{{c_{\alpha_F}}}}\right)\right],
\end{align*}
which can be assumed to be finite by choosing an appropriate initial condition. Recall indeed that the choice of $P_0$ in the construction of the sequence $({\mathcal{Q}}_T)_{T \geq 0}$ is free. 
{Therefore, assuming that 
\begin{equation}\label{3}
	\int_{\mathcal{P}(\R^d)} \exp \bigl( 2 c_2 v (\widetilde m)^{1/2} \bigr) P_0(\mathrm{d}\widetilde m) < +\infty, 
\end{equation}
and}
 using lower-semi-continuity of the function $P \mapsto \langle \exp(2c_2 v(\cdot)^{1/2}{)} \: ; \: P\rangle_{\mathcal{P}(\R^d)}$, we get that 
\begin{equation}\label{eq:}
	\int_{\mathcal{P}(\R^d)} \exp\left( 2c_2 v(\widetilde m)^{1/2} \right) \widebar P(\mathrm{d}\widetilde m) \leq \lim\inf_{T \to +\infty} \int_{\mathcal{P}(\R^d)} \exp\left( 2c_2 v(\widetilde m)^{1/2} \right) \mathcal{Q}_T(\mathrm{d}\widetilde m) < +\infty,
\end{equation}
which is \eqref{eq:finite2}. 
\vskip 6pt
\noindent In order to prove that $\widebar P \in \mathcal{P}_2(\mathcal{P}_2(\R^d))$, we use the same approach and 
{get} for any $T > 0$, 
\begin{equation}
\label{eq:cesaro:T:mu_2}
\begin{split}
	\int_{\mathcal{P}(\R^d)} \mu_2(\widetilde m) \mathcal{Q}_T(\mathrm{d}\widetilde m) & = \frac{1}{T} \int_0^T\int_{\mathcal{P}(\R^d)} \mu_2(\widetilde m) P_t(\mathrm{d}\widetilde m) \mathrm{d}t \\
	& = \frac{1}{T} \int_0^T \E_0[\mu_2(m_t)]\mathrm{d}t.
\end{split}
\end{equation} 
Now, we claim that, {under 
\eqref{3} 
and}
Assumptions~\ref{A1} and~\ref{A2},
\begin{equation*}
	\sup_{t\geq 0} \E_0[\:\mu_2(m_t)] < +\infty. 
\end{equation*}
Indeed, there {exist} two positive constants $b_1, b_2 >0$, such that for all $(x,y) \in \R^d$,
\begin{equation*}
	(x-y)\cdot \left(G(x) - G(y)\right) \leq -b_1|x-y|^2 + b_2.
\end{equation*}
Then, we get 
\begin{align*}
	\frac{\mathrm{d}}{\mathrm{d}t}\E_0[\mu_2(m_t)] \leq & -2b_1\E_0[\mu_2(m_t) ]+ 2b_2 + (\sigma^2 + \sigma_0^2) \\
	& + 2|G(0)|\E_0[\mu_2(m_t)]^{1/2} + 2C_F\E_0\left[ \mu_2(m_t)^{1/2}v(m_t)^{1/2}\right].
\end{align*} 
Using once again {\eqref{eq:v:mt} and Jensen's inequality} to control the last term, we obtain
\begin{align*}
	\frac{\mathrm{d}}{\mathrm{d}t}\E_0[\mu_2(m_t)]&\leq -2b_1\E_0[\mu_2(m_t) ]+ 2\left( |G(0)| + C_F {v(m_0)^{1/2}} + \frac{C_F {\sigma}}{{\sqrt{c_{\alpha_F}}}}\right) \E_0\left[ \mu_2(m_t)\right]^{1/2} 
	\\
	&\quad + 2b_2 + d(\sigma^2 + \sigma_0^2)
	\\
	&{\leq - b_1\E_0[\mu_2(m_t) ]+ b_1^{-1} \left( |G(0)| + C_F {v(m_0)^{1/2}} + \frac{C_F {\sigma}}{{\sqrt{c_{\alpha_F}}}}\right)^2    + 2b_2 + d(\sigma^2 + \sigma_0^2)},
\end{align*}
and conclude using Grönwall's Lemma that $\sup_{t\geq 0} \E_0[\:\mu_2(m_t)] < +\infty.$
Finally, using the same lower-semi-continuity argument {as in \eqref{eq:} and returning to \eqref{eq:cesaro:T:mu_2}}, we get
\begin{equation}\label{eq:control:2:}
	\int_{\mathcal{P}(\R^d)} \mu_2(\widetilde m) \widebar{P}(\mathrm{d}\widetilde m) \leq \lim_{T \to +\infty} \int_{\mathcal{P}(\R^d)} \mu_2(\widetilde m) \mathcal{Q}_T(\mathrm{d}\widetilde m) < +\infty. 
\end{equation}
The latter  implies that $\widebar P \in \mathcal{P}_2(\mathcal{P}_2(\R^d))$. 
\end{proof}

\subsection{{Completion of the proof of Theorem~\ref{th:main}}}
We are now ready for the proof of Theorem~\ref{th:main} which is a straightforward combination of Proposition~\ref{prop:p1} and Lemma~\ref{lemma:estimate}.

\begin{proof}[Proof of Theorem~\ref{th:main}]
	We know from Proposition~\ref{prop:p1} that for any $\alpha_F \geq \max(L_G, m_G)$ and any $\sigma_0 > 0$, there {exist} $\widebar \sigma > 0$ (depending on $F, G, \sigma_0$ and $d$) such that for all $\sigma \leq \widebar \sigma$ and for 
	 any $1$-Lipschitz continuous  function $\phi : \mathcal{P}_1(\R^d) \rightarrow {\mathbb R}$
	and all 
	$m,\widetilde m \in \mathcal{P}_2(\R^d)$, 
	\begin{equation*}
	\left|\mathscr{P}_t\phi(m) - \mathscr{P}_t\phi(\widetilde m)\right| \leq \mathcal{K}(m, \widetilde m) e^{- \mathfrak{c}t},
	\quad {t \geq 0,}
\end{equation*}
for some $ \mathfrak{c} > 0$ independent of $m$ and $\widetilde m$.  Let us now consider {an invariant measure} $\widebar P \in \mathcal{P}_2(\mathcal{P}_2(\R^d))$, {\color{black} with a finite conditional exponential moments}, {whose} existence is ensured by Lemma~\ref{lemma:estimate}. 
{Then,
\begin{equation}
\begin{aligned}
\label{eq:proof:thm1:1}
	\left|\mathscr{P}_t\phi(m) - \int_{\mathcal{P}(\R^d)}\mathscr{P}_t\phi(\widetilde m) \widebar P(\mathrm{d}\widetilde m)\right| &\leq \int_{\mathcal{P}(\R^d)}|\mathscr{P}_t\phi(m) - \mathscr{P}_t\phi(\widetilde m) |\widebar P(\mathrm{d}\widetilde m) 
	\\
	&\leq e^{-  \mathfrak{c}t}\int_{\mathcal{P}(\R^d)}\mathcal{K}(m, \widetilde m) \widebar P(\mathrm{d}\widetilde m).  
\end{aligned}
\end{equation}

Now, coming back to 
Proposition~\ref{prop:p1}
for the shape of $\mathcal{K}(m, \widetilde m)$, we obtain that
\begin{equation*}
\begin{split}
\int_{\mathcal{P}(\R^d)}\mathcal{K}(m, \widetilde m) \widebar P(\mathrm{d}\widetilde m)
&\leq C_d \mu_2(m)^{1/2} 
\exp\left(c_2 v(m)^{1/2} \right)
\int_{\mathcal{P}(\R^d)}
\exp\left(c_2 v(m)^{1/2} \right)
\widebar P(\mathrm{d}\widetilde m)
\\
&\quad +  C_d \exp\left(c_2 v(\widetilde m)^{1/2} \right)
\int_{\mathcal{P}(\R^d)}
\mu_2(\widetilde m)^{1/2} 
\exp\left(c_2 v(\widetilde m)^{1/2} \right)
\widebar P(\mathrm{d}\widetilde m),
\end{split}
\end{equation*}
with $c_2 := {c_1^{\alpha_F}}/{c_{\alpha_F}}$.}

{Using Hölder inequality and leveraging on Lemma~\ref{lemma:estimate}
we have
\begin{equation*}
\begin{split}
\int_{\mathcal{P}(\R^d)}\mathcal{K}(m, \widetilde m) \widebar P(\mathrm{d}\widetilde m)
&\leq C_d \bigl( 1+ \mu_2(m)^{1/2} \bigr) \exp\left(c_2 v( m)^{1/2} \right). 
\end{split}
\end{equation*}
Substituting in 
\eqref{eq:proof:thm1:1}, we complete the proof.}
\end{proof}

\section{Proof of Proposition~\ref{prop:p1}}\label{proofP1}
This section is dedicated to the proof of Proposition \ref{prop:p1}. 

\subsection{Ansatz}
{The strategy of proof of Proposition~\ref{prop:p1} relies on 
the auxiliary Proposition~\ref{prop:approx}.
This latter result says that the residual conditional variance of 
$X$ (solution
to Equation \eqref{LP}) is small with $\sigma$. In turn, it prompts us to focus on the process $\E_1[X]=(\E_1[X_t])_{t \geq 0}$}. The latter has the following {dynamics}:
\begin{equation*}
	\mathrm{d}\E_1[X_t] = \E_1[G(X_t)] \mathrm{d}t + \E_1[F(X_t - \E_1[X_t])]\mathrm{d}t + \sigma_0\mathrm{d}B^0_t, 
	\quad {t \geq 0}. 
\end{equation*}
{Whilst this is not a closed equation (in $\E_1[X_t]$), our approach is to write it as a perturbation of 
a finite-dimensional stochastic differential equation.}
Heuristically, we have (thanks to Proposition~\ref{prop:approx}) that on a large time scale
\begin{align}
	& \E_1[F(X_t - \E_1[X_t])] \approx \sigma^2/c_{\alpha_F},\label{eq:approx:1}
	\\
	& \E_1[G(X_t)] = G(\E_1[X_t]) + \left[ \E_1[G(X_t)] - G(\E_1[X_t]) \right] \approx G(\E_1[X_t]) + \sigma^2/c_{\alpha_F}\label{eq:approx:2}.
\end{align}
Equations \eqref{eq:approx:1} and \eqref{eq:approx:2} are {rather informal, and we refrain from giving here a rigorous meaning to the symbol $\approx$ in both of them}. However, 
{these two expansions suggest} that the {dynamics} of $(\E_1[X_t])_{t\geq 0}$ {indeed} resembles to the one of the diffusion process 
$Y${, constructed on $(\Omega_0, \mathbb{F}^0, \mathbb{P}_0)$ as the solution of} 
\begin{equation}\label{eq:diffusion}
	\mathrm{d} Y_t = G(Y_t)\mathrm{~d}t + \sigma_0\mathrm{d}B^0_t, \quad {t \geq 0 \ ; \quad Y_0 = {\mathbb E}_1(X_0).} 
\end{equation}
{Here, it is worth emphasizing that, even obvious, 
\eqref{eq:diffusion} is not a McKean-Vlasov but a mere diffusion equation. 
In particular, the long time analysis 
of
\eqref{eq:diffusion} falls within a much wider literature, since the study of the long time behaviour of diffusion processes has been an important topic of interest}, see \cite{ane2000inegalite, bakry2008simple, bakry1985diffusions} to name just a few. 
{Below, we make use of coupling arguments, in the spirit of Eberle \cite{eberle2016reflection}, 
who stated contraction properties for the law of the solution of \eqref{eq:diffusion} by 
 building on ideas developed earlier in \cite{lindvall1986coupling}.} The next section gives some details on this {approach} and the contraction results {available in this} framework. 

\subsection{Contraction for classical diffusion processes}\label{subsec:contraction} 
 Following \cite{eberle2016reflection}, we introduce the following quantities :
\begin{align*}
R_0 & := \inf \{R \geq 0: \kappa(r) \geq 0, \quad \forall r \geq R\}, \\
R_1 & := \inf \left\{R \geq R_0: \kappa(r) R\left(R-R_0\right) \geq 8, \quad  \forall r \geq R\right\},
\end{align*}
so that $R_0 \leq R_1$. Thanks to  Assumption~\ref{A1}, $R_0$ and $R_1$ are finite. {\color{black}Moreover}, let
\begin{align*}
& \varphi(r):=\exp \left(-\frac{1}{4} \int_0^r s \kappa(s)^{-} \mathrm{d} s\right), \quad \Phi(r){:=}\int_0^r \varphi(s) \mathrm{d} s, \\
& g(r):=1-  {\left(\int_0^{r \wedge R_1} \Phi(s)\varphi(s)^{-1}   \mathrm{d} s \right)}\left(2\int_0^{R_1} \Phi(s)\varphi(s)^{-1} \mathrm{d} s\right)^{-1} ,
\end{align*}
and 
\begin{equation}\label{eq:f}
f(r):=\int_0^r \varphi(s) g(s) \mathrm{d}s.
\end{equation}
Then, we have that
\begin{itemize}
\item $\varphi$ is non-increasing, $\varphi(0)=1$, and $\varphi(r)=\varphi\left(R_0\right)$ for any $r \geq R_0$;
~\\
\item $g$ is non-increasing, $g(0)=1$, and $g(r)=1/2$ for any $r \geq R_1$;
~\\
\item $f^\prime$ is bounded on $[0, +\infty)$;
~\\
\item $f$ is concave, $f(0)=0, f^{\prime}(0)=1$, and
\begin{equation*}
\Phi(r) / 2 \leq f(r) \leq \Phi(r) \quad \text { for any } r \geq 0 .
\end{equation*}
\end{itemize}
The last display combined with the fact that for all $r \geq 0$, $\Phi(r) / 2 \geq \varphi(R_0)r/2$ and $\Phi(r) \leq r$, implies  \begin{equation}\label{eq:equivalence}
	\kappa_1 r \leq f(r) \leq r,
\end{equation}
with $\kappa_1 := \varphi(R_0)/2$. {This allows us to} define the   {equivalent} deformation of $1$-Wasserstein distance
\begin{equation*}
\mathrm{d}^{\R^d}_f(\mu, \nu) {:=}\inf_{\pi \in \Pi(\mu, \nu)}\int f(|x-y|) \pi(\mathrm{~d}x, \mathrm{~d}y){,}
\end{equation*}
for {any two} probability measures $\mu, \nu$ on $\mathbb{R}^d$, where the infimum is taken over all couplings $\pi$ of $\mu$ and $\nu$. {Last but not least}, the function $f$ we constructed satisfies 
\begin{equation}\label{eq:ineq}
	f^{\prime \prime}(r)-\frac{1}{4} r \kappa(r) f^{\prime}(r) \leq-\frac{c}{2\sigma_0^2} f(r) \quad \text { for  all } r>0,
\end{equation}
and for some $c > 0$. Such concentration property for the function $f$ is classical, and the proof of inequality \eqref{eq:ineq} can be found in \cite{conforti2023coupling,  durmus2020elementary, eberle2016reflection}. Leveraging on the contraction property of $f$ given by Equation \eqref{eq:ineq}, we get the following result
{(as a direct consequence of 
\cite{eberle2016reflection})}:
\begin{theorem}[Corollary 2 in \cite{eberle2016reflection}]\label{th:eberle}
	Under Assumption~\ref{A1} {and with} $\sigma_0 > 0$, there exists a constant $c := c_{\sigma_0} > 0$, such that for any pair $(Y, \widetilde Y)$ {of solutions to} Equation \eqref{eq:diffusion} with different initial conditions, we have
	\begin{equation}\label{eq:contra:f}
		\mathrm{d}^{\R^d}_f\left(\mathcal{L}_1(Y_t), \mathcal{L}_1(\widetilde Y_t)\right) \leq \mathrm{d}^{\R^d}_f\left(\mathcal{L}_1(Y_0), \mathcal{L}_1(\widetilde Y_0)\right)e^{-ct}, 
		\quad {t \geq 0}.  
	\end{equation}
\end{theorem}

\begin{remark}
	{Thanks} to  \eqref{eq:equivalence}, we get that inequality \eqref{eq:contra:f} holds (up to some constant) when replacing $\mathrm{d}^{\R^d}_f$ by the classical $\mathrm{d}^{\R^d}_1$ Wasserstein distance. 
\end{remark}

The idea of the proof of Theorem~\ref{th:eberle} is to use the reflection coupling introduced in \cite{lindvall1986coupling}
and widely used in the literature \cite{conforti2023coupling, durmus2020elementary, eberle2016reflection,eberle2019sticky}. The classical reflection coupling
{is constructed by multiplying the noise by a symmetry matrix mapping instantaneously $Y_t - \widetilde Y_t$ on 
its opposite}, see \cite[Section 3]{lindvall1986coupling} for details. {In this framework, the distance between 
the two processes $Y$ and $\widetilde Y$ evolves according to an It\^o process with constant noise intensity. The hope is thus to benefit 
from the presence of the noise to obtain further recurrent properties that prevent the processes from staying away from one another. For sure, one cannot expect 
this intuition to hold true for any type of drift; in order to guarantee a relevant form of recurrence, some further confining properties are necessary.} 

{When dealing with the stochastic differential equation 
\eqref{eq:diffusion}, 
the condition
		$\lim\sup_{r \to +\infty}\kappa(r) > 0 $
		required in Assumption 		
		\ref{A1} is key in the verification of the latter confining properties. In itself, this is not an obvious result: one must indeed keep in mind that $G$ is not strictly decreasing, meaning that $\kappa$ may be negative on some part of the space. Clearly, things become even more subtle 
		in presence of the conditional McKean-Vlasov interaction. 
		Still, one would like to settle down a similar reflection argument for Equation \eqref{LP}, which raises preliminary questions of measurability. 
		Briefly, in order to preserve the 
	 form of	the common noise, one must consider a reflection matrix that is measurable with respect to the latter
	 (and not to the idiosyncratic noise).} The idea developed in the next subsection is to introduce a coupling $(X, \widetilde X)$
	 {of 
	 solutions to 
	 \eqref{LP} in which the reflection is achieved  by reflecting with respect to $\E_1[X_t] - \E_1[\widetilde X_t]$. However, while reflecting with respect to the conditional 
	 expectations may seem appropriate to our objective, it cannot 
	 suffice on its own to get a proof of Proposition~\ref{prop:p1} that would be a mere copy and paste of the argument of 
	  \cite{eberle2016reflection}. The difficulty that one meets when adapting  \cite{eberle2016reflection} to \eqref{LP} is clear: 
	  the processes that solve \eqref{LP} are not 
	those that are used in the reflection. 
	This is where Proposition~\ref{prop:approx} comes in and this is one of the main innovation of our work: 
	we use the fact that the distance between the solutions to  \eqref{LP} and their conditional expectations (which are also the processes 
	entering the reflection) is small with $\sigma$. 
	Mathematically, the difficulty is to revisit the whole machinery of \cite{eberle2016reflection}  by using the fact that, although it is not zero, 
	the latter distance is 
	small. }	
	%

\subsection{Proof of Proposition~\ref{prop:p1}}
To {ease} the reading of this subsection, {we briefly outline the strategy of proof}:\\

\noindent \textit{Proof Outline.}
\begin{itemize}
	\item \textit{Step 1}. We construct a coupling $(X,\widetilde X)$ inspired by the reflection coupling introduced in \cite{lindvall1986coupling} (as explained in subsection~\ref{subsec:contraction}) and adapted to the presence of {a} common noise. 
	\item \textit{Steps 2, 3 \& 4}. We prove that {this} coupling allows the conditional expectations $\E_1[X]$ and $\E_1[\widetilde X]$ 
	to get closer in time, up to some residual {error} that depends on {the distance} between the centred processes $X - \E_1[X]$ and $\widetilde X- \E_1[\widetilde X]$. 
	{This is established thanks to the confining property of} $G$ guaranteed by Assumption~\ref{A1}.	
	\item \textit{Step 5}. In parallel (the key point {is} to run the two arguments at the same time), we 
	{prove the mirror result, namely we show that the centred processes $X - \E_1[X]$ and $\widetilde X- \E_1[\widetilde X]$
	get closer in time, up to some residual {error} that depends on {the distance} between the
	conditional expectations $\E_1[X]$ and $\E_1[\widetilde X]$.
	This is established thanks to the strictly contracting properties of the interaction term} $F$ guaranteed by Assumption~\ref{A2}.
	\item \textit{Step 6}. We combine together the results of Step 4 and Step 5 to conclude. 
\end{itemize} 

\begin{proof}
	\textit{Step 1}. Let us consider a bounded and $\mathrm{d}_1^{\R^d}$-Lipschitz continuous function $\phi : \mathcal{P}_1(\R^d) \to \R$. In particular there exists $\| \phi\|_{\mathrm{lip}} > 0$ such that for any $m,\widetilde m \in \mathcal{P}_1(\R^d)$, 
	\begin{equation*}
		|\phi(m) - \phi(\widetilde m)| \leq \| \phi\|_{\mathrm{lip}} \mathrm{d}_1^{\R^d}(m, \widetilde m). 
	\end{equation*}
	Let us then fix $m, \widetilde m \in \mathcal{P}_2(\R^d)$ together with two independent random variables $X_0, \widetilde X_0$, defined on 
	{$(\Omega_1, \mathbb{F}^1, \mathbb{P}_1)$} such that $\mathcal{L}_1(X_0) = m$ and  $\mathcal{L}_1(\widetilde X_0) = \widetilde m$. Then, we {extend $X_0$ and $\widetilde X_0$ to the product space $(\Omega, \mathbb{F}, \mathbb{P})$ in a natural manner}.  
	
\begin{remark}\label{rk:IC}
	It is important to notice here that, even if the initial conditions $X_0$, $\widetilde{X}_0$ are defined on the product space $(\Omega, \mathbb{F}, \mathbb{P})$, they are independent of {${\mathcal F}^0$}. 
\end{remark}
	
For a given $\delta > 0$, {let} us now introduce the following {coupling} of two solutions of the SDE \eqref{LP} with initial conditions $X_0$ and $\widetilde X_0$:
\begin{align}\label{eq:coupling}
    \left\{
    \begin{array}{ll}
        \begin{aligned}
            \mathrm{d}X^\delta_t = G(X^\delta_t)\mathrm{d}t + &F\left(X^\delta_t- \E_1[X^\delta_t]\right)\mathrm{d}t  + \sigma \mathrm{d}B_t 
            + \sigma_0\left\{ \pi_{\delta}(E^\delta_t)\mathrm{d}B^0_t + \lambda_\delta(E^\delta_t)\mathrm{d}\widetilde{B}^0_t\right\},
        \end{aligned}
        \\
        \\
        \begin{aligned}
            \mathrm{d}\widetilde{X}^\delta_t = G(\widetilde{X}^\delta_t)\mathrm{d}t + & F\left(\widetilde{X}^\delta_t- \E_1[\widetilde{X}^\delta_t]\right)\mathrm{d}t  + \sigma \mathrm{d}B_t \\
	            &+ \sigma_0\left\{ \pi_{\delta}(E^\delta_t)\left( \mathrm{Id} - 2 e^\delta_t\left(e^\delta_t\right)^\top \right)\mathrm{d}B^0_t + \lambda_\delta(E^\delta_t)\mathrm{d}\widetilde{B}^0_t\right\},
        \end{aligned}
    \end{array}
    \right.
\end{align}
where 
\begin{itemize}
	\item {the} initial conditions are given by $X^\delta_0 = X_0$ and $\widetilde{X}^\delta_0 = \widetilde{X}_0$;
	\item $\widetilde B^0$ is a $d$-dimensional Brownian motion defined on {$(\Omega_0, \mathcal{F}^0, \mathbb{P}_0)$} adapted to the filtration 
	{$\mathbb{F}^0$} and independent of $B^0$;
	\item {for} all $t \geq 0$, $E^\delta_t := \E_1[X^\delta_t] - \E_1[\widetilde{X}^\delta_t]$, and 
	\begin{equation*}
		e^\delta_t := \begin{cases}
			E^\delta_t/|E^\delta_t| \quad \text{if} \quad |E^\delta_t| \neq 0\\
			0 \quad \quad \: \text{otherwise} ;
		\end{cases}
	\end{equation*}
	\item {the} function $\pi_\delta$ is Lipschitz continuous on $\R^d$ with value in $[0,1]$, such that for any $x \in \R^d$,
	\begin{equation*}
		\pi_\delta(x) = \begin{cases}
			0 \quad \text{if} \quad |x| \leq \delta/2\\
			1 \quad \text{if} \quad |x| \geq \delta.
		\end{cases}
	\end{equation*}
	\item {the} function $\lambda_\delta$ is Lipschitz-continuous and satisfies $\lambda_\delta^2(x) = 1 - \pi_\delta^2(x)$, for all $x \in \R^d$. 
\end{itemize}
The first thing to notice is that under the Lipschitz continuity assumptions on $G$ and $F$, and the integrability properties of the initial conditions, we have existence and uniqueness of {a} weak solution of \eqref{eq:coupling}. 
Moreover, thanks to the choice of $\pi_\delta$ and $\lambda_\delta$, the pair $(X^\delta, \widetilde X^\delta)$ is a coupling of $(Z, \widetilde Z)$ where for $Z$ and $\widetilde Z$ admits the same dynamic,
\begin{equation}
	\mathrm{d}Z_t = G(Z_t) \mathrm{d}t + F(Z_t - \E_1[Z_t])\mathrm{d}t + \sigma \mathrm{d}B_t + \sigma_0\mathrm{d}B^0_t,
\end{equation}
with initial condition $Z_0 = X_0$ and $\widetilde Z_0 = \widetilde X_0$. 
This is mainly due to Levy's characterisation of the Brownian motion which is applicable here as for all $x \in \R^d$, $\lambda^2_\delta(x) + \pi_\delta^2(x) = 1$. Details on the coupling method are given in subsection~\ref{subsec:contraction}.
Below, we focus on the long-time behavior of 
\begin{equation*}
\E_0\big[ \E_1\big[|X^\delta_t - \widetilde X^\delta_t|\big] \big], 
\end{equation*}
for a given $\delta > 0$. 
~\\

\textit{Step 2.}  Letting $A^\delta_t := X^\delta_t - \E_1[X^\delta_t]$ and $\widetilde A^\delta_t := \widetilde X^\delta_t - \E_1[\widetilde X^\delta_t]$, we get that for all $t\geq 0$, 
\[|X^\delta_t - \widetilde X^\delta_t| \leq |\E_1[X^\delta_t] - \E_1[\widetilde X^\delta_t]| + |A^\delta_t - \widetilde A^\delta_t|.\] The idea of the proof is to control both $|\E_1[X^\delta_t] - \E_1[\widetilde X^\delta_t]|$ and $|A^\delta_t - \widetilde A^\delta_t|$ simultaneously. This will allow us to prove that both terms {converge} in fact to 0 (in some sense). Let us begin with the difference of the conditional expectations:
\begin{align*}
	\mathrm{d}\left(\E_1[X^\delta_t] - \E_1[\widetilde X^\delta_t]\right) = &  \E_1[G(X^\delta_t) - G(\widetilde X^\delta_t)]\mathrm{d}t + \E_1[F(A^\delta_t) - F(\widetilde A^\delta_t)] \mathrm{d}t \\
	 + &  2\sigma_0\pi_\delta(E^\delta_t)  e^\delta_te^\delta_t\cdot\mathrm{d}B^0_t.
\end{align*}
Using Itô's formula and recalling that $E^\delta_t = \E_1[X^\delta_t] - \E_1[\widetilde X^\delta_t]$ and $e^\delta_t\cdot E^\delta_t = |E^\delta_t|$, we obtain
\begin{align*}
	\mathrm{d}\left|\E_1[X^\delta_t] - \E_1[\widetilde X^\delta_t]\right|^2   =  & \:  2\left(\E_1[X^\delta_t] - \E_1[\widetilde X^\delta_t]\right)\cdot\left( \E_1[G(X^\delta_t)] - \E_1[G(\widetilde X^\delta_t)]\right)\mathrm{d}t \\
	+ & \: 2\left(\E_1[X^\delta_t] - \E_1[\widetilde X^\delta_t]\right)\cdot\left( \E_1[F(A^\delta_t)] - \E_1[F(\widetilde A^\delta_t)]\right)\mathrm{d}t \\
	+& \: 4\sigma_0\left|\E_1[X^\delta_t] - \E_1[\widetilde X^\delta_t]\right|\pi_\delta(E^\delta_t)e^\delta_t\cdot\mathrm{d}B^0_t \\
	 + & \: 2\sigma_0^2\pi^2_\delta(E^\delta_t)\mathrm{d}t.
\end{align*}
This combined with the particular shape of the reflection, and noting that $(\int_0^te^\delta_s\cdot\mathrm{d}B^0_s)_t$ is a 1-dimensional Brownian motion tank's to Levy's characterisation allows us to write
\begin{equation}\label{eq:abs}
\begin{aligned}
	\mathrm{d}\left|\E_1[X^\delta_t] - \E_1[\widetilde X^\delta_t]\right| = & \: e^\delta_t \cdot\left( \E_1[G(X^\delta_t)] - \E_1[G(\widetilde X^\delta_t)]\right)\mathrm{d}t \\
	 + & \: e^\delta_t \cdot\left( \E_1[F(A^\delta_t)] - \E_1[F(\widetilde A^\delta_t)]\right)\mathrm{d}t \\
	 + &\: 2\sigma_0\pi_\delta(E^\delta_t)e^\delta_t\cdot \mathrm{d}B^0_t, \quad {t \geq 0}. 
\end{aligned}
\end{equation}
Deriving Equation \eqref{eq:abs} is not straightforward and can be done using similar techniques as in \cite{durmus2020elementary}.  The underlying concept is that the presence of $\pi_\delta$ keeps the process $E^\delta$ from lingering near zero. {The complete proof of \eqref{eq:abs}} can be found in Appendix~\ref{Appendix}.
We now recall (see Assumption~\ref{A1}) {that} $\kappa$ is a continuous function on $(0, \infty)$ satisfying
\begin{equation}\label{eq:kappa}
\liminf _{r \rightarrow \infty} \kappa(r)>0 \quad \text { and } \quad \int_0^1 r \kappa(r)^{-} \mathrm{d} r<\infty .
\end{equation}
Then, using Itô's formula once again, we obtain
\begin{align*}
	\mathrm{d}f\left(\left|\E_1[X^\delta_t] - \E_1[\widetilde X^\delta_t]\right|\right) = & \: f^\prime\left(\left|\E_1[X^\delta_t] - \E_1[\widetilde X^\delta_t]\right|\right)e^\delta_t \cdot\left( \E_1[G(X^\delta_t)] - \E_1[G(\widetilde X^\delta_t)]\right)\mathrm{d}t \\
	+ & \: f^\prime\left(\left|\E_1[X^\delta_t] - \E_1[\widetilde X^\delta_t]\right|\right)e^\delta_t \cdot\left( \E_1[F(A^\delta_t)] - \E_1[F(\widetilde A^\delta_t)]\right)\mathrm{d}t \\
	+ & \: 2\sigma_0f^\prime\left(\left|\E_1[X^\delta_t] - \E_1[\widetilde X^\delta_t]\right|\right)\pi_\delta(E^\delta_t)e^\delta_t\cdot \mathrm{d}B^0_t \\
	+ & \: 2\sigma_0^2f^{\prime\prime}\left(\left|\E_1[X^\delta_t] - \E_1[\widetilde X^\delta_t]\right|\right)\pi^2_\delta(E^\delta_t) \mathrm{d}t. 
\end{align*}
\trash{
\begin{remark}
	Here, it is important to notice that there is no loss of generality to work with the function $f$ as Equation \eqref{eq:equivalence} implies that the distance $(x,y) \mapsto f(|x-y|)$ is equivalent to the distance generated by the Euclidian norm on $\R^d$.
\end{remark}
}
{Then}, we rewrite
\begin{equation}
\begin{aligned}\label{eq:f:E}
	\mathrm{d}f\left(\left|\E_1[X^\delta_t] - \E_1[\widetilde X^\delta_t]\right|\right) = & \: f^\prime\left(\left|\E_1[X^\delta_t] - \E_1[\widetilde X^\delta_t]\right|\right)e^\delta_t \cdot\left( G(\E_1[X^\delta_t]) - G(\E_1[\widetilde X^\delta_t])\right)\mathrm{d}t \\
	 + & \: 2\sigma_0^2f^{\prime\prime}\left(\left|\E_1[X^\delta_t] - \E_1[\widetilde X^\delta_t]\right|\right)\pi^2_\delta(E^\delta_t) \mathrm{d}t \\
	+ & \: 2\sigma_0f^\prime\left(\left|\E_1[X^\delta_t] - \E_1[\widetilde X^\delta_t]\right|\right)\pi_\delta(E^\delta_t)e^\delta_t\cdot \mathrm{d}B^0_t  + r^1_t\mathrm{d}t,
\end{aligned}
\end{equation}
where for all $t \geq 0$, $r_t^1 := \widebarroman{I}_1(t) + \widebarroman{II}_1(t) + \widebarroman{III}_1(t)$, with 

\begin{equation}\label{eq:r1}
\begin{aligned}
	 &\widebarroman{I}_1(t) = f^\prime\left( \left| E^\delta_t\right|\right)e^\delta_t\cdot\E_1\Big[ \int_0^1 \{ DG(\rho X^\delta_t + (1-\rho)\widetilde X^\delta_t) - DG(\rho\E_1[X^\delta_t] + (1-\rho)\E_1[\widetilde X^\delta_t])\}\: \mathrm{d}\rho
	\\
	&\hspace{300pt} \:(\E_1[X^\delta_t] - \E_1[\widetilde{X}^\delta_t])\Big] \: ; 
	\\
	 &\widebarroman{II}_1(t) =    f^\prime\left( \left| E^\delta_t\right|\right)e^\delta_t\cdot\E_1\Big[ \int_0^1 \{ DG(\rho X^\delta_t + (1-\rho)\widetilde X^\delta_t) - DG(\rho\E_1[X^\delta_t] + (1-\rho)\E_1[\widetilde X^\delta_t])\}\: \mathrm{d}\rho
	 \\
	&\hspace{300pt}
	 \:(A^\delta_t - \widetilde A^\delta_t)\Big] \: ; 
	\\
	 &\widebarroman{III}_1(t) = f^\prime\left( \left| E^\delta_t\right|\right)e^\delta_t\cdot\E_1\Big[ \int_0^1  DF(\rho A^\delta_t + (1-\rho)\widetilde A^\delta_t) \: \mathrm{d}\rho\:(A^\delta_t - \widetilde A^\delta_t)\Big],
\end{aligned} 
\end{equation}
where we used the following two identities: $\E_1[X^\delta_t] - \E_1[\widetilde{X}^\delta_t]+
A^\delta_t - \widetilde A^\delta_t=X^\delta_t - \widetilde X^\delta_t$
and
$\E_1[A^\delta_t] - \E_1[\widetilde{A}^\delta_t]=0$.\color{black}
~\\

\textit{Step 3}. 
{Recall \eqref{eq:ineq}:}
\begin{equation*}
	f^{\prime \prime}(r)-\frac{1}{4} r \kappa(r) f^{\prime}(r) \leq-\frac{c}{2\sigma_0^2} f(r) \quad \text { for  all } r>0.
\end{equation*}
Combining the {above equation} with the fact that, from Assumption~\ref{A1},
\begin{equation*}
	(x-y)\cdot(G(x) - G(y)) \leq -\frac{2}{\sigma_0^2}\kappa(|x-y|)|x-y|^2, \quad {x,y \in \R^d}, 
\end{equation*}
 we obtain 
\begin{equation}\label{eq:f:1}
\begin{aligned}
	& f^\prime\left(\left|\E_1[X^\delta_t] - \E_1[\widetilde X^\delta_t]\right|\right)e^\delta_t \cdot\left( G(\E_1[X^\delta_t]) - G(\E_1[\widetilde X^\delta_t])\right) \\
	& \hspace{200pt}   + 2\sigma_0^2f^{\prime\prime}\left(\left|\E_1[X^\delta_t] - \E_1[\widetilde X^\delta_t]\right|\right)\pi^2_\delta(E^\delta_t) \\
	& \leq -\frac{\sigma_0^2}{2} f^\prime\left(\left|E^\delta_t\right|\right)|E^\delta_t|\kappa\left(\left| E^\delta_t\right|\right) + 2\sigma_0^2f^{\prime\prime}\left(\left| E^\delta_t\right|\right)\pi_\delta^2(E^\delta_t)  \\
	& \leq  2\sigma_0^2\left( -\frac{1}{4} f^\prime\left(\left|E^\delta_t\right|\right)|E^\delta_t|\kappa\left(\left| E^\delta_t\right|\right) + f^{\prime\prime}\left(\left| E^\delta_t\right|\right)\right)\pi^2_\delta(E^\delta_t)  -\frac{\sigma_0^2}{2} f^\prime\left(\left|E^\delta_t\right|\right)|E^\delta_t|\kappa\left(\left| E^\delta_t\right|\right)\lambda_\delta^2(E^\delta_t)\\
	& \leq -c f\left(\left|\E_1[X^\delta_t] - \E_1[\widetilde X^\delta_t]\right|\right) + h(\delta),
\end{aligned}
\end{equation}
\normalsize
where
\begin{equation*}
	h(\delta) := \frac{\sigma_0^2}{2} f^\prime\left(\left|E^\delta_t\right|\right)|E^\delta_t|\kappa\left(\left| E^\delta_t\right|\right)\lambda_\delta^2(E^\delta_t).
\end{equation*}
Now, recalling that $\lambda_\delta(x) = 0$ as soon as $|x| \geq \delta$, we obtain for $\delta\in [0,1]$,

\begin{equation}
\label{eq:h:delta}
	|h(\delta)| \leq \frac{\sigma_0^2}{2} \|f^\prime\|_{\infty}\delta \kappa_2,
\end{equation}
\color{black}
where $\kappa_2 = \sup_{0 \leq r \leq 1} \kappa(r)$, from which we deduce that $h(\delta) \to 0$ with $\delta$.  Then, back to \eqref{eq:f:E}, we can write
\begin{align*}
	\mathrm{d}f\left(\left|\E_1[X^\delta_t] - \E_1[\widetilde X^\delta_t]\right|\right) \leq & -c f\left(\left|\E_1[X^\delta_t] - \E_1[\widetilde X^\delta_t]\right|\right)\mathrm{d}t + |r^1_t|\mathrm{d}t + h(\delta)\mathrm{d}t
	\\
	& + 2\sigma_0f^\prime\left(\left|\E_1[X^\delta_t] - \E_1[\widetilde X^\delta_t]\right|\right)\pi_\delta(E^\delta_t)e^\delta_t\cdot \mathrm{d}B^0_t,
\end{align*}
and 
\begin{equation*}
	\frac{\mathrm{d}}{\mathrm{d}t} \E_0\left[ f\left(\left|\E_1[X^\delta_t] - \E_1[\widetilde X^\delta_t]\right|\right)\right] \leq  -c \: \E_0\left[ f\left(\left|\E_1[X^\delta_t] - \E_1[\widetilde X^\delta_t]\right|\right)\right] +\E_0 [|r^1_t |] +
	{\E_0}{\left[|
	 h(\delta) | \right]}. 
\end{equation*}
Using Grönwall's Lemma, we obtain

\begin{equation}\label{eq:contraction:expectations}
\begin{split}
	\E_0\left[ f\left(\left|\E_1[X^\delta_t] - \E_1[\widetilde X^\delta_t]\right|\right)\right] &\leq \E_0\left[ f\left(\left|\E_1[X^\delta_0] - \E_1[\widetilde X^\delta_0]\right|\right)\right]e^{-ct} + \int_0^t e^{-c(t-s)} \E_0 [|r^1_s |] \mathrm{d}s 
	\\
	&\quad + \frac{h(\delta)}{c}. 
	\end{split}
\end{equation}
\color{black}
\\

\textit{Step 4}. It remains to control the three parts in the expansion of the residual term $r^1_\cdot$, see  \eqref{eq:r1}. For the two first ones, we use the Lipschitz continuity of $DG$ with respect to the norm $\triple \cdot\triple$ and write
\begin{align*}
	& |\widebarroman{I}_1(t)| \leq  C_G\E_1\left[|X^\delta_t - \E_1[X^\delta_t]|^2 + |\widetilde X^\delta_t - \E_1[\widetilde X^\delta_t]|^2\right]^{1/2}\left|\E_1[X^\delta_t] - \E_1[\widetilde X^\delta_t]\right| \: ; \\
	& |\widebarroman{II}_1(t)| \leq  C_G\E_1\left[|X^\delta_t - \E_1[X^\delta_t]|^2 + |\widetilde X^\delta_t - \E_1[\widetilde X^\delta_t]|^2\right]^{1/2}\E_1\left[|A^\delta_t - \widetilde A^\delta_t|^2\right]^{1/2}. 
\end{align*}
Letting $\eta^\delta_t = \E_1\left[|X^\delta_t - \E_1[X^\delta_t]|^2 + |\widetilde X^\delta_t - \E_1[\widetilde X^\delta_t]|^2\right]^{1/2}$, for all $t \geq 0$, and recalling from Proposition~\ref{prop:approx} {that}
\begin{equation*}
	\eta^\delta_t \leq \eta^\delta_0e^{-{c_{\alpha_F}}t} + 2\sigma { c_{\alpha_F}^{-1/2}}\sqrt{d},
\end{equation*}
we obtain
\begin{equation}\label{eq:control:E:1a2}
\begin{split}
	&|\widebarroman{I}_1(t)| + |\widebarroman{II}_1(t)| 
	\\
	&\quad \leq C_G\left(
	 \eta^\delta_0e^{-{c_{\alpha_F}}t} + 2\sigma  c_{\alpha_F}^{-1/2}\sqrt{d}
	 \right)\left(\E_1\left[|A^\delta_t - \widetilde A^\delta_t|^2\right]^{1/2} + \left|\E_1[X^\delta_t] - \E_1[\widetilde X^\delta_t]\right|\right).
	\end{split}
\end{equation}
For the term $\widebarroman{III}_1(t)$, we use the fact that $\E_1[A^\delta_t] = \E_1[\widetilde A^\delta_t] = 0$ and write 
\begin{equation}\label{eq:control:E:3}
\begin{aligned}
	|\widebarroman{III}_1(t)| = & \left|\E_1\left[ \int_0^1 \{ DF(\rho A^\delta_t + (1-\rho)\widetilde A^\delta_t) - DF(0)\}\mathrm{d}\rho\cdot  (A^\delta_t - \widetilde A^\delta_t) \right]\right| \\
	\leq & C_F\left( \eta^\delta_0e^{-{c_{\alpha_F}}t} + 2\sigma { c_{\alpha_F}^{-1/2}}\sqrt{d}
	\right)\E_1\left[|A^\delta_t - \widetilde A^\delta_t|^2\right]^{1/2},
\end{aligned}
\end{equation}
where we used once again Proposition~\ref{prop:approx} together with the Lipschitz continuity of $F$ to get the last line. Then, combining \eqref{eq:f:E}, \eqref{eq:control:E:1a2} and \eqref{eq:control:E:3}, we have
\begin{align*}
	\E_0[|r^1_t|]
	 & \leq \E_0\left[ |\widebarroman{I}_1(t)| + |\widebarroman{II}_1(t)| + |\widebarroman{III}_1(t)|\right] \\
	&  \leq (C_G + C_F )\left(  \eta^\delta_0e^{-{c_{\alpha_F}}t} + 2\sigma c_{\alpha_F}^{-1/2}\sqrt{d}\right){\E_0\left[\E_1\left[|A^\delta_t - \widetilde A^\delta_t|^2\right]^{1/2} + \left|\E_1[X^\delta_t] - \E_1[\widetilde X^\delta_t]\right|\right]}.
\end{align*}
Here, we recall that $\eta^\delta_0$ is deterministic because $m$ and $\widetilde m$ are deterministic, see Remark~\ref{rk:IC}. 
Finally, using inequalities
\eqref{eq:equivalence} and \eqref{eq:contraction:expectations}, we obtain 
\begin{equation}
\begin{aligned}\label{eq:difference:moment}
	 & \kappa_1 \E_0\big[|\E_1[X^\delta_t] -  \E_1[\widetilde X^\delta_t]|\big] \leq  e^{-ct}\left|\E_1[X^\delta_0] - \E_1[\widetilde X^\delta_0]\right| { + \frac{h(\delta)}{c} }
	 \\
	& + C_b\int_0^t { e^{-c(t-s)}} \left(  \eta^\delta_0e^{-{c_{\alpha_F} s}} + 2\sigma { c_{\alpha_F}^{-1/2}}\sqrt{d} \right)
{	\E_0\left[\E_1\left[|A^\delta_s - \widetilde A^\delta_s|^2\right]^{1/2} 
+ |\E_1[X^\delta_s] - \E_1[\widetilde X^\delta_s]|\right]} \mathrm{d}s,
\end{aligned}
\end{equation}
where $C_b := C_G + C_F $. 
~\\

\textit{Step 5}. We now focus on the {(normed)} difference 
{process $(|A^\delta_t - \widetilde A^\delta_t|)_{t \geq 0}$}. 
{Thanks to the properties of $F$, as stated in Assumption~\ref{A2}, we are here able to show that
it decreases exponentially fast  to zero.} 

{Recalling from 
Step 2 that 
$A^\delta_t = X^\delta_t - \E_1[X^\delta_t]$ for $t \geq 0$, we have}
\begin{equation*}
	\mathrm{d}A^\delta_t = \left(G(X^\delta_t) - \E_1[G(X^\delta_t)]\right)\mathrm{d}t + \left(F(X^\delta_t - \E_1[X^\delta_t]) \mathrm{d}t - \E_1[F(X^\delta_t - \E_1[X^\delta_t]]\right)\mathrm{d}t + \sigma\mathrm{d}B_t,
\end{equation*}
{from which we obtain} 
\begin{align*}
	\frac{\mathrm{d}}{\mathrm{d}t}(A^\delta_t - \widetilde A^\delta_t) = & \left[G(X^\delta_t) - G(\widetilde X^\delta_t)\right] + \left[F(A^\delta_t) - F(\widetilde A^\delta_t)\right]\\
	& - \left[\E_1[G(X^\delta_t)] - \E_1[G(\widetilde X^\delta_t)] \right]- \left[\E_1[F(A^\delta_t)] - \E_1[F(\widetilde A^\delta_t)]\right]. 
\end{align*}
\trash{It is crucial to note that we avoided reflecting on the idiosyncratic noise, which has led to a deterministic dynamic for the difference in the centered processes. This decision might seem like a disadvantage, especially when considering works such as \cite{durmus2020elementary}, where reflecting on the idiosyncratic noise enhances the contraction rate of the conditional expectation process. However, our focus diverges here since we are primarily concerned with scenarios where \(\sigma\) is small. Moreover, getting a deterministic dynamic for \(A^\delta_\cdot - \widetilde A^\delta_\cdot\) streamlines subsequent calculations and helps to understand the significant influence of the interaction term.}
Then,
\begin{equation}\label{eq:A:dynamic}
\begin{aligned}
	\frac{\mathrm{d}}{\mathrm{d}t}|A^\delta_t - \widetilde A_t^{\delta}|^2 & = 2(A^\delta_t - \widetilde A^\delta_t)\cdot(G(X_t^\delta)- G(\widetilde X^\delta_t)) \\
	& + 2(A^\delta_t - \widetilde A^\delta_t)\cdot(F(A_t^\delta)- F(\widetilde A^\delta_t)) \\
	& -2(A^\delta_t - \widetilde A^\delta_t)\cdot(\E_1[G(X_t^\delta)] - \E_1[G(\widetilde X_t^\delta)]) \\
	& - 2(A^\delta_t - \widetilde A^\delta_t)\cdot(\E_1[F(A^\delta_t)] - \E_1[F(\widetilde A^\delta_t)])
\end{aligned}
\end{equation}
We can immediately leverage on Assumption~\ref{A2} to control the second term on the right hand side:
\begin{equation}\label{eq:control:A:F}
	2(A^\delta_t - \widetilde A^\delta_t)\cdot(F(A_t^\delta)- F(\widetilde A^\delta_t)) \leq -2\alpha_F|A^\delta_t - \widetilde A^\delta_t|^2.
\end{equation}
Let us now focus on the terms involving $G$. We {have}
\begin{align*}
	&(A^\delta_t - \widetilde A^\delta_t)\cdot\left(G(X_t^\delta)- G(\widetilde X^\delta_t) -\E_1[G(X_t^\delta)] +\E_1[G(\widetilde X_t^\delta)]\right) 
	\\
	& = (A^\delta_t - \widetilde A^\delta_t)\cdot
	\biggl[ \left( \int_0^1 \left\{ DG(\rho X_t^\delta + (1-\rho)\widetilde X^\delta_t)-DG(\rho\E_1[X_t^\delta] + (1-\rho)\E_1[\widetilde X^\delta_t])\right\}\mathrm{d}\rho\right)(X^\delta_t - \widetilde X^\delta_t) \biggr] \nonumber
	\\
	& + (A^\delta_t - \widetilde A^\delta_t)\cdot \biggl[ \left( \int_0^1DG(\rho\E_1[X_t^\delta] + (1-\rho)\E_1[\widetilde X^\delta_t])\mathrm{d}\rho\right)(X^\delta_t - \widetilde X^\delta_t) \biggr] \nonumber
	\\
	& - (A^\delta_t - \widetilde A^\delta_t)\cdot\E_1\left[\left( \int_0^1 \left\{ DG(\rho X_t^\delta + (1-\rho)\widetilde X^\delta_t)-DG(\rho\E_1[X_t^\delta] + (1-\rho)\E_1[\widetilde X^\delta_t])\right\}\mathrm{d}\rho\right)(X^\delta_t - \widetilde X^\delta_t) \right] 
	\nonumber
	\\
	& - (A^\delta_t - \widetilde A^\delta_t)\cdot \biggl[ \left( \int_0^1DG(\rho\E_1[X_t^\delta] + (1-\rho)\E_1[\widetilde X^\delta_t])\mathrm{d}\rho\right)(\E_1[X^\delta_t] - \E_1[\widetilde X^\delta_t]) \biggr]. \nonumber
\end{align*}
{By combining} the second and fourth term in the right hand side, {we obtain}
\begin{align*}
	 &(A^\delta_t - \widetilde A^\delta_t)\cdot\left(G(X_t^\delta)- G(\widetilde X^\delta_t)) -\E_1[G(X_t^\delta)] +\E_1[G(\widetilde X_t^\delta)]\right)  
	 \\
	 = & (A^\delta_t - \widetilde A^\delta_t)
	 \cdot \biggl[ \left( \int_0^1 \left\{ DG(\rho X_t^\delta + (1-\rho)\widetilde X^\delta_t)-DG(\rho\E_1[X_t^\delta] + (1-\rho)\E_1[\widetilde X^\delta_t])\right\}\mathrm{d}\rho\right)(X^\delta_t - \widetilde X^\delta_t) \biggr] 
	 \\
	 + & (A^\delta_t - \widetilde A^\delta_t)\cdot\left( \int_0^1DG(\rho\E_1[X_t^\delta] + (1-\rho)\E_1[\widetilde X^\delta_t])\mathrm{d}\rho\right)(A^\delta_t - \widetilde A^\delta_t) 
	 \\
	 - & (A^\delta_t - \widetilde A^\delta_t)
	 \cdot\E_1\left[\left( \int_0^1 \left\{ DG(\rho X_t^\delta + (1-\rho)\widetilde X^\delta_t)-DG(\rho\E_1[X_t^\delta] + (1-\rho)\E_1[\widetilde X^\delta_t])\right\}\mathrm{d}\rho\right)(X^\delta_t - \widetilde X^\delta_t) \right]
	 \\
	 =:& \widebarroman{A}_1(t)+ \widebarroman{A}_2(t) + \widebarroman{A}_3(t).
\end{align*}
Then, if we perform the following decomposition of the first term on the right-hand side, we obtain
\begin{equation}
\label{eq:AA1+AA2}
	\widebarroman{A}_1(t) =  \widebarroman{AA}_1(t)+\widebarroman{AA}_2(t),
\end{equation}
with
\begin{align*}
	& \widebarroman{AA}_1(t) := (A^\delta_t - \widetilde A^\delta_t)\cdot \biggl[ \biggl( \int_0^1 \Bigl\{ DG(\rho X_t^\delta + (1-\rho)\widetilde X^\delta_t) -DG(\rho\E_1[X_t^\delta] + (1-\rho)\E_1[\widetilde X^\delta_t])\Bigr\}\mathrm{d}\rho\biggr)\\
	&\hspace*{320pt} (\E_1[X^\delta_t] - \E_1[\widetilde X^\delta_t]) \biggr] \\
	& \widebarroman{AA}_2(t) :=   (A^\delta_t - \widetilde A^\delta_t)\cdot\Big( \int_0^1 \Big\{ DG(\rho X_t^\delta + (1-\rho)\widetilde X^\delta_t) \\
	& \hspace*{190pt}-DG(\rho\E_1[X_t^\delta] + (1-\rho)\E_1[\widetilde X^\delta_t])\Big\}\mathrm{d}\rho\Big)(A^\delta_t - \widetilde A^\delta_t).
\end{align*}

Then, combining the terms
${\widebarroman{A}_2(t)}$
and ${\widebarroman{AA}_2(t)}$, we get
\begin{equation}
\label{eq:decomposition:AA1:A2:AA2:A3}
\begin{split}
	 &(A^\delta_t - \widetilde A^\delta_t)\cdot\left(G(X_t^\delta)- G(\widetilde X^\delta_t) -\E_1[G(X_t^\delta)] +\E_1[G(\widetilde X_t^\delta)]\right) 
	 \\
	 &=  { \widebarroman{AA}_1(t) + \left[ \widebarroman{A}_2(t) +  \widebarroman{AA}_2(t)\right] + 
	  \widebarroman{A}_3(t)}
	 \\
	&=: {\widebarroman{AA}_1(t) + \widebarroman{II}_2(t) + \widebarroman{A}_3(t)},
	\end{split} 
\end{equation}
{where $\widebarroman{AA}_1(t)$ and 
$\widebarroman{A}_3(t)$ have been already defined and 
$\widebarroman{II}_2(t)$ is defined by}   
\begin{align*}
	& \widebarroman{II}_2(t) := (A^\delta_t - \widetilde A^\delta_t)\cdot\left( \int_0^1DG(\rho X_t^\delta + (1-\rho)\widetilde X^\delta_t)\mathrm{d}\rho\right)(A^\delta_t - \widetilde A^\delta_t).
\end{align*}
We now control $\widebarroman{II}_2(t)$ {as we handled} Equation \eqref{eq:control:A:F}. We get 
\begin{equation}
\label{eq:control:II_2}
	{\widebarroman{II}_2(t)} \leq m_G|A^\delta_t - \widetilde{A}^\delta_t|^2, 
\end{equation}
{where we recall from 
\eqref{eq:lowerboundkappa}
that $m_G \in \R$ stands for the upper bound
of the function $r  \mapsto [ - \sigma_0^2 \kappa(r)/2]$ introduced in 
Assumption~\ref{A1}. Indeed, as a consequence of \ref{A1},} one can show that for any $x,y\in \R^d$, 
\begin{equation*}
	y \cdot \bigl(DG(x)y \bigr) \leq m_G|y|^2.
\end{equation*}
{
Then, 
inserting \eqref{eq:control:A:F}, 
\eqref{eq:decomposition:AA1:A2:AA2:A3}
and 
\eqref{eq:control:II_2}
into \eqref{eq:A:dynamic}, we obtain} 
\begin{equation*}
\begin{split}
	&\frac{\mathrm{d}}{\mathrm{d}t}|A^\delta_t - \widetilde A_t^{\delta}|^2 
	\\
	&\leq  -2({\alpha_F}- m_G)|A^\delta_t - \widetilde A^\delta_t|^2 + \widebarroman{AA}_1(t) + \widebarroman{A}_3(t) +  2(A^\delta_t - \widetilde A^\delta_t)\cdot(\E_1[F(A^\delta_t)] - \E_1[F(\widetilde A^\delta_t)]).
	\end{split}
\end{equation*}
{By taking expectation under $\mathbb{P}_1$ and by using 
the fact that 
$\E_1[A^\delta_t] = \E_1[\widetilde A^\delta_t] = 0$, we get rid the last two terms 
in the above right-hand side. We obtain}
\begin{equation}\label{eq:control:2}
	\frac{\mathrm{d}}{\mathrm{d}t}\E_1\left[|A^\delta_t - \widetilde A_t^{\delta}|^2\right] \leq -2(\alpha_F-m_G)\E_1\left[|A^\delta_t - \widetilde A_t^{\delta}|^2\right] + {2} \E_1\left[ \widebarroman{AA}_1(t)\right]. 
\end{equation}
In order to control $\widebarroman{AA}_1(t)$, we 
{come back to 
\eqref{eq:AA1+AA2} and}
use the fact that $\E_1[X^\delta_t] - \E_1[\widetilde X^\delta_t]$ is measurable with respect to the common noise. 
{By 
Cauchy-Schwarz' inequality,} we obtain
\begin{equation*}
	\E_1\left[ \widebarroman{AA}_1(t)\right] \leq C_G\E_1[|A^\delta_t - \widetilde A^\delta_t|^2]^{1/2} \eta^\delta_t |\E_1[X^\delta_t] - \E_1[\widetilde X^\delta_t]|,
\end{equation*}
where we recall that $\eta^\delta_t = \E_1[|A^\delta_t|^2 + |\widetilde{A}^\delta_t|^2]^{1/2}$. Then, from Proposition~\ref{prop:approx}, we get
\begin{equation*}
	\E_1\left[ \widebarroman{AA}_1(t)\right] \leq \E_1[|A^\delta_t - \widetilde A^\delta_t|^2]^{1/2} \left( \eta^\delta_0 e^{-c_{\alpha_F}t} + \sigma c_{\alpha_F}^{-1/2}\sqrt{d}\right) |\E_1[X^\delta_t] - \E_1[\widetilde X^\delta_t]|,
\end{equation*}
and  \eqref{eq:control:2} rewrites
\begin{equation}
\begin{aligned}\label{eq:before:le2}
	\frac{\mathrm{d}}{\mathrm{d}t}\E_1\left[|A^\delta_t - \widetilde A_t^{\delta}|^2\right] \leq & -2(\alpha_F-m_G)\E_1\left[|A^\delta_t - \widetilde A_t^{\delta}|^2\right] \\
	&+ 2\E_1[|A^\delta_t - \widetilde A^\delta_t|^2]^{1/2} \left( \eta^\delta_0 e^{-c_{\alpha_F}t} + \sigma c_{\alpha_F}^{-1/2}\sqrt{d}\right) |\E_1[X^\delta_t] - \E_1[\widetilde X^\delta_t]|. 
\end{aligned}
\end{equation}
Then, we get 
\begin{equation}
\begin{aligned}\label{eq:sqrt}
	\frac{\mathrm{d}}{\mathrm{d}t} \E_1\left[ |A^\delta_t - \widetilde A^\delta_t|^2\right]^{1/2} \leq & -({\alpha_F}- m_G)\E_1\left[ |A^\delta_t - \widetilde A^\delta_t|^2\right]^{1/2} 
	\\
	& + C_G(\eta^\delta_0e^{-c_{\alpha_F}t} + \sigma c_{\alpha_F}^{-1/2}\sqrt{d})|\E_1[X^\delta_t] - \E_1[\widetilde X^\delta_t]|,
\end{aligned}
\end{equation}
see Appendix~\ref{app:A2} for details. Using Grönwall's Lemma, we obtain
\begin{align*}
	\E_1[|A^\delta_t - \widetilde{A}^\delta_t|^2]^{1/2} \leq \ & \E_1[|A^\delta_0 - \widetilde{A}^\delta_0|^2]^{1/2}\exp\left( -({\alpha_F}- m_G) t\right) \\
	& + C_G\int_0^t e^{ -({\alpha_F}- m_G)(t-s)}(\eta^\delta_0e^{-c_{\alpha_F}s} + \sigma c_{\alpha_F}^{-1/2}\sqrt{d})|\E_1[X^\delta_s] - \E_1[\widetilde X^\delta_s]| \mathrm{d}s. 
\end{align*}
Finally, taking again expectation $\E_0$, 
\begin{equation}\label{eq:difference:variance}
\begin{aligned}
	&\E_0\left[\E_1[|A^\delta_t -  \widetilde A^\delta_t|^2]^{1/2}\right] \leq  \exp(-({\alpha_F}- m_G)t)\E_1[|A^\delta_0 - \widetilde A^\delta_0|^2]^{1/2} \\
	& + C_G\int_0^t
	{e^{-({\alpha_F}- m_G)(t-s)}
	(\eta^\delta_0 e^{-c_{\alpha_F}s} + \sigma c_{\alpha_F}^{-1/2}\sqrt{d})} 
	\E_0\left[{|\E_1[X^\delta_s] - \E_1[\widetilde X^\delta_s]|}\right] \mathrm{d}s.
\end{aligned}
\end{equation}
~\\

\textit{Step 6.} Now, as \eqref{eq:difference:moment} remains true replacing the rate $c$ by $\widetilde c \leq c$, we can assume without loss of generality that ${\alpha_F}> m_G + c$. We can then combine Equations \eqref{eq:difference:moment} and \eqref{eq:difference:variance} to obtain, for all $t \geq 0$,
\begin{equation}\label{eq:finaleq:delta}
\begin{aligned}
	& {\E_0\left[\E_1\left[|A^\delta_t - \widetilde A^\delta_t|^2\right]^{1/2}+|\E_1[X^\delta_t] - \E_1[\widetilde X^\delta_t]|\right] }
	\\
	&\leq  \left(\E_1\left[|A^\delta_0 - \widetilde A^\delta_0|^2\right]^{1/2}  + |\E_1[X^\delta_0] - \E_1[\widetilde X^\delta_0]|\right)e^{-ct} 
	 + \frac{h(\delta)}{c}
	\\
	& + c_3^b\int_0^t 
	 { e^{-c(t-s)}} \left(  \eta^\delta_0e^{-{c_{\alpha_F} s}} + 2\sigma { c_{\alpha_F}^{-1/2}}\sqrt{d} \right)
	{\E_0\left[\E_1\left[|A^\delta_s - \widetilde A^\delta_s|^2\right]^{1/2}+
	|\E_1[X^\delta_s] - \E_1[\widetilde X^\delta_s]|\right] } \mathrm{d}s,
\end{aligned}
\end{equation}
for some constant $c_3^b> 0$ that depends on $F$ and $G$. 

 With the notation $\Theta^\delta_t := \E_1\left[|A^\delta_t - \widetilde A^\delta_t|^2\right]^{1/2} + |\E_1[X^\delta_t] - \E_1[\widetilde X^\delta_t]|$,
Equation \eqref{eq:finaleq:delta} reads
\begin{equation*}
	\E_0[\Theta_t^\delta] \leq \Theta_0^\delta e^{-ct} 
	 + \frac{h(\delta)}{c} + c_3^b\int_0^t 
	 { e^{-c(t-s)}} \left(  \eta^\delta_0e^{-{c_{\alpha_F} s}} + 2\sigma { c_{\alpha_F}^{-1/2}}\sqrt{d} \right)
	{\E_0\left[\Theta^\delta_s\right] } \mathrm{d}s.
\end{equation*}
Then, using Grönwall Lemma, we obtain
\begin{equation}\label{eq:G:delta}
\begin{aligned}
	e^{ct}{\E_0\left[\Theta^\delta_t\right] } \leq & \left(\Theta_0 + \frac{h(\delta)}{c}\right)\exp\left(\frac{c_3^b}{c_{\alpha_F}}\eta_0 + \frac{c_3^b\sigma\sqrt{d}}{\sqrt{c_{\alpha_F}}}t\right) \\
	& + h(\delta)\int_0^t e^{cs}\exp\left(\frac{c_3^b}{c_{\alpha_F}}\eta_0 + \frac{c_3^b\sigma\sqrt{d}}{\sqrt{c_{\alpha_F}}}(t-s)\right)\mathrm{d}s.
\end{aligned}
\end{equation}
We emphasize here that we write $\Theta_0$ and $\eta_0$ without reference to $\delta$ because these quantities are predetermined and independent of $\delta$. Now, let us notice that creating the couplage we defined two Brownian motions adapted to the filtration $\mathbb{F}_0$:
\begin{align*}
	 \beta^\delta_t &= \int_0^t \pi_\delta(E^\delta_s)( \mathrm{Id}  - 2e^\delta_s(e^\delta_s)^\top)\mathrm{d}B^0_s + \int_0^t \lambda_\delta(E^\delta_s)\mathrm{d}\widetilde B^0_s\\
 	 \widetilde\beta^\delta_t &= \int_0^t \pi_\delta(E^\delta_s)\mathrm{d}B^0_s + \int_0^t \lambda_\delta(E^\delta_s)\mathrm{d}\widetilde B^0_s.
\end{align*}
Setting $m^\delta_t = \mathcal{L}_1(X^\delta_t)$ et $\widetilde m^\delta_t = \mathcal{L}_1(\widetilde X^\delta_t)$, we get
\begin{align*}
	& \mathrm{d}_t  m^\delta_t = \nabla \cdot \left(\frac{\sigma^2 + \sigma_0^2}{2} \nabla m^\delta_t - m^\delta_t \left( G + F(\cdot - \mu_1(m^\delta_t))\right)\right)\mathrm{d}t -\sigma_0 \nabla m^\delta_t \cdot \mathrm{d}\beta^\delta_t,\\
	& \mathrm{d}_t  \widetilde m^\delta_t = \nabla \cdot \left(\frac{\sigma^2 + \sigma_0^2}{2} \nabla \widetilde m^\delta_t - \widetilde m^\delta_t \left( G + F(\cdot - \mu_1(\widetilde m^\delta_t))\right)\right)\mathrm{d}t -\sigma_0 \nabla \widetilde m^\delta_t \cdot \mathrm{d}\widetilde \beta^\delta_t.
\end{align*}
Then, one has $\mathscr{P}_t\phi(m) = \E_0\left[ \phi(m^\delta_t)| m^\delta_0 = m\right]$ and $\mathscr{P}_t\phi(\widetilde m) = \E_0\left[ \phi(\widetilde m^\delta_t)| \widetilde m^\delta_0 = \widetilde m\right]$. Then, we obtain that for any $\delta > 0$, 
\begin{equation}
\begin{aligned}\label{eq:SG}
	\left|\mathscr{P}_t\phi(m) - \mathscr{P}_t\phi(\widetilde m)\right| & \leq  \| \phi\|_{\mathrm{lip}} \E_0\left[\E_1\left[\left|X^\delta_t - \widetilde X^\delta_t\right|\right]\right]\\
	& \leq \| \phi\|_{\mathrm{lip}} \E_0[\Theta^\delta_t]
\end{aligned}
\end{equation}
Plugging the estimate given by Equation \eqref{eq:G:delta} into Equation \eqref{eq:SG} gives
\begin{align*}
	\left|\mathscr{P}_t\phi(m) - \mathscr{P}_t\phi(\widetilde m)\right| \leq & \|\phi\|_{\mathrm{lip}}\left(\Theta_0 + \frac{h(\delta)}{c}\right)\exp\left(\frac{c_3^b}{c_{\alpha_F}}\eta_0 + \left(\frac{c_3^b\sigma\sqrt{d}}{\sqrt{c_{\alpha_F}}}-c\right)t\right) \\
	& + \|\phi\|_{\mathrm{lip}}h(\delta)\int_0^t e^{-c(t-s)}\exp\left(\frac{c_3^b}{c_{\alpha_F}}\eta_0 + \frac{c_3^b\sigma\sqrt{d}}{\sqrt{c_{\alpha_F}}}(t-s)\right)\mathrm{d}s.
\end{align*}
The latter being true for all $\delta > 0$, we can consider the limit $\delta \to 0$ to finally obtain
\begin{equation*}
\left|\mathscr{P}_t\phi(m) - \mathscr{P}_t\phi(\widetilde m)\right| \leq \|\phi\|_{\mathrm{lip}}\Theta_0 \exp\left(\frac{c_3^b}{c_{\alpha_F}}\eta_0 + \frac{c_3^b\sigma\sqrt{d}}{\sqrt{c_{\alpha_F}}}t -ct\right).
\end{equation*}  
Then, {for $\sigma$ and $\alpha$ satisfying $\mathfrak{c} := \sigma c_3^bc_{\alpha_F}^{-1/2} - c < 0$}, 
we get the expected exponential rate of convergence. Moreover, writing
\begin{align*}
	\Theta_0\exp\left( \frac{c_3^b}{c_{\alpha_F}}\eta_0\right) = & \left(\E_1\left[|A_0 - \widetilde A_0|^2\right]^{1/2}  + |\E_1[X_0] - \E_1[\widetilde X_0]|\right)\exp\left(\frac{c_3^b}{c_{\alpha_F}}\eta_0\right) \\
	\leq & C_2\left( \mu_2(m)^{1/2} + \mu_2(\widetilde m)^{1/2} \right) \exp\left( \frac{c_3^b}{c_{\alpha_F}}(v(m)^{1/2} + v(\widetilde m)^{1/2})\right),
\end{align*}
{for a new constant positive $C_2$, 
we complete the proof of Proposition~\ref{prop:p1}.}
\end{proof}

\subsection*{Acknowledgments} 
Fran\c{c}ois Delarue acknowledges the financial support of the European Research Council (ERC) under the European Union’s Horizon 2020 research and innovation program (AdG ELISA project, Grant agreement No. 101054746). Rapha\"{e}l Maillet wishes to thank the Universit\'{e} C\^{o}te d'Azur for its hospitality during the preparation of this work, along with Pierre Cardaliaguet for fruitful discussions.

\trash{
\color{blue}
Check the assumptions: 
\begin{itemize}
	\item $a(x) = \sigma_0^2$;
	\item By convexity assumption of the potential $V$ at infinity and as $\nabla V$ is assumed to be Lipschitz continuous, we know that $\nabla V$ behaves like $|x|$ at infinity. Then, 
	\begin{equation*}
		\beta_1 |x|^2 - C_1 \leq \nabla V(x)\cdot x \leq \beta_2 |x|^2 - C_2,
	\end{equation*}
	with $\beta_1 > 0$. 
\end{itemize}
$\Rightarrow$ The previous two points imply $(A1)$ and $(A2)$ of \cite{ichihara2013large}. 
Moreover, $DG$ is bounded and then $(A3')$ is satisfied. 
\color{black}}

\bibliographystyle{amsplainhyper_m}
\bibliography{biblio}


\appendix
\section{}\label{Appendix}
\subsection{Proof of Equation \eqref{eq:abs}}
\begin{proof}
Using Itô's formula, and recalling that $E^\delta = X^\delta - \widetilde X^\delta$, we get that for any $\delta > 0$,
	\begin{align*}
		\mathrm{d}|\mathbb{E}_1[X^{\delta}_t] - \mathbb{E}_1[\widetilde X^{\delta}_t]|^2 =  &\quad 2(\mathbb{E}_1[X^{\delta}_t] - \mathbb{E}_1[\widetilde X^{\delta}_t])\cdot\left(\mathbb{E}_1[G(X^{\delta}_t)] - \mathbb{E}_1[G(\widetilde X^{\delta}_t)]\right)\mathrm{d}t \\
		&+ 2(\mathbb{E}_1[X^{\delta}_t] - \mathbb{E}_1[\widetilde X^{\delta}_t])\cdot\left(\mathbb{E}_1[F(X^{\delta}_t - \E_1[X^\delta_t])] - \mathbb{E}_1[G(F(\widetilde X^{\delta}_t - \E_1[\widetilde X^\delta_t]))]\right)\mathrm{d}t \\
		& + 4\sigma_0\pi_\delta(E^\delta_t)((\mathbb{E}_1[X^{\delta}_t] - \mathbb{E}_1[\widetilde X^{\delta}_t])\cdot e^\delta_t)e^\delta_t\cdot\mathrm{d}B^0_t \\
		& + 4\sigma_0^2\pi^2_\delta(E^\delta_t)\mathrm{d}t. 
	\end{align*}
For any $\varepsilon > 0$, let us introduce $\psi_\varepsilon : [0,+\infty] \in r \mapsto (r+\varepsilon)^{1/2}$. Since this function is twice continuously differentiable, we can write
\begin{align}\label{eq:psi}
	\mathrm{d}\psi_{\varepsilon}\left(|\mathbb{E}_1[X^{\delta}_t] - \mathbb{E}_1[\widetilde X^{\delta}_t]|^2\right) = & \quad 2 \psi^\prime_{\varepsilon}(|E^\delta_t|^2) E^\delta_t\cdot\left(\mathbb{E}_1[G(X^{\delta}_t)] - \mathbb{E}_1[G(\widetilde X^{\delta}_t)]\right)\mathrm{d}t\nonumber \\
	& + 2 \psi^\prime_{\varepsilon}(|E^\delta_t|^2) E^\delta_t\cdot\left(\mathbb{E}_1[F(X^{\delta}_t - \E_1[X^\delta_t)] - \mathbb{E}_1[F(\widetilde X^{\delta}_t - \E_1[\widetilde X^\delta_t)]\right)\mathrm{d}t\nonumber \\
	& + 4\sigma_0 \psi^\prime_{\varepsilon}(|E^\delta_t|^2)\pi_\delta(E^\delta_t) (E^\delta_t \cdot e^\delta_t)e^\delta_t\cdot\mathrm{d}B^0_t \\
	& + 4 \sigma_0^2\psi^\prime_{\varepsilon}(|E^\delta_t|^2)\pi^2_\delta(E^\delta_t)\mathrm{d}t \nonumber\\
	& + 8\sigma_0^2\psi^{\prime\prime}_{\varepsilon}(|E^\delta_t|^2)|E^\delta_t|^2\pi_\delta(E^\delta_t)\mathrm{d}t. \nonumber 
\end{align}
We now want to take the limit $\varepsilon\to 0$. Using dominated convergence theorem and stochastic dominated convergence theorem as stated in \cite[Theorem 2.12]{revuz2013continuous}, combined with the bound $4r\psi^\prime_{\varepsilon}(r^2) \leq 1$, we can deal with the first two lines and get that for all $t\geq 0$,
\begin{equation}\label{eq:TCD1}
	\lim_{\varepsilon \to 0}\int_0^t 2 \psi^\prime_{\varepsilon}(|E^\delta_s|^2) E^\delta_s\cdot\left(\mathbb{E}_1[G(X^{\delta}_s)] - \mathbb{E}_1[G(\widetilde X^{\delta}_s)]\right)\mathrm{d}s = \int_0^t e^\delta_s\cdot\left(\mathbb{E}_1[G(X^{\delta}_s)] - \mathbb{E}_1[G(\widetilde X^{\delta}_s)]\right)\mathrm{d}s , 
\end{equation}
\begin{equation}
\begin{aligned}\label{eq:TCD4}
	& \lim_{\varepsilon \to 0}\int_0^t 2 \psi^\prime_{\varepsilon}(|E^\delta_s|^2) E^\delta_s\cdot\left(\mathbb{E}_1[F(X^{\delta}_t - \E_1[X^\delta_t])] - \mathbb{E}_1[F(\widetilde X^{\delta}_t - \E_1[\widetilde X^\delta_t])]\right)\mathrm{d}s \\
	&= \int_0^t e^\delta_s\cdot\left(\mathbb{E}_1[F(X^{\delta}_t - \E_1[X^\delta_t])] - \mathbb{E}_1[F(\widetilde X^{\delta}_t - \E_1[\widetilde X^\delta_t])]\right)\mathrm{d}s ,
\end{aligned}
\end{equation}
and 
\begin{equation}\label{eq:TCD2}
	\lim_{\varepsilon \to 0} \int_0^t 4\sigma_0 \psi^\prime_{\varepsilon}(|E^\delta_s|^2)  \pi_\delta(E^\delta_s)(E^\delta_s\cdot e^\delta_s) e^\delta_s \cdot \mathrm{d}B^0_s =  \int_0^t 2\sigma_0 \pi_\delta(E^\delta_s)e^\delta_s\cdot\mathrm{d}B^0_s,
\end{equation}
almost surely. 
For the last two terms in Equation \eqref{eq:psi}, let us remark that for all $r \geq 0$, 
\begin{equation}
	\psi^\prime_{\varepsilon}(r) + 2\psi^{\prime\prime}(r) = \frac{1}{2(r +\varepsilon)^{1/2}} - \frac{r}{2(r + \varepsilon)^{3/2}} \to 0,
\end{equation}
when $\varepsilon$ tends to 0. 
In order to we need to properly take the limit $\varepsilon \to 0$ in Equation \eqref{eq:psi}, we need to take advantage of the presence of the function $\pi_\delta$ for $\delta > 0$. In Indeed, we have  
\begin{align*}
	& \left|4 \sigma_0^2\psi^\prime_{\varepsilon}(|E^\delta_t|^2)\pi^2_\delta(E^\delta_t) + 8\sigma_0^2\psi^{\prime\prime}_{\varepsilon}(|E^\delta_t|^2)|E^\delta_t|^2\pi_\delta(E^\delta_t)\right| \\
	& \leq \left|\pi^2_\delta(E^\delta_t)\sigma_0^2\left( 4\psi^\prime_{\varepsilon}(|E^\delta_t|^2) + 8\psi^{\prime\prime}_{\varepsilon}(|E^\delta_t|^2)|E^\delta_t|^2\right)\right|. 
\end{align*}
Moreover, we know that $\psi^\prime_{\varepsilon}(r^2) + 2\psi^{\prime\prime}_{\varepsilon}(r^2)r^2 \leq r^{-3}$, for all $r > 0$ and $\varepsilon \leq 1$. Using the presence of $\pi_\delta$, we have that the integrand is null near 0. Then, we can once again apply dominated convergence theorem and obtain
\begin{equation}\label{eq:TCD3}
	\lim_{\varepsilon \to 0} \int_0^t 4 \left\{\sigma_0^2\psi^\prime_{\varepsilon}(|E^\delta_s|^2)\pi_\delta(E^\delta_s)^2+ 8\sigma_0^2\psi^{\prime\prime}_{\varepsilon}(|E^\delta_s|^2)|E^\delta_s|^2\psi_\delta(E^\delta_s) \right\}\mathrm{d}s = 0. 
\end{equation}
Finally, combining Equations \eqref{eq:TCD1}, \eqref{eq:TCD4} \eqref{eq:TCD2} and \eqref{eq:TCD3}, we get the expected result. 
\end{proof}

\subsection{Proof of Equation \eqref{eq:sqrt}} \label{app:A2}
\begin{proof}
To prove Equation, let us introduce the following technical Lemma:
\begin{lemma}\label{le:2}
	For any positive and differentiable function $u : [0, +\infty) \to [0, +\infty)$ such that
	\begin{equation*}
		\frac{\mathrm{d}}{\mathrm{d}t} u(t) \leq 2k(t)\sqrt{u(t)},
	\end{equation*}
	for some function $k : [0, +\infty) \to \R$, we have
	\begin{equation*}
		\frac{\mathrm{d}}{\mathrm{d}t} \sqrt{u(t)} \leq k(t).
	\end{equation*}
\end{lemma}
\begin{proof}[Proof of Lemma~\ref{le:2}]
	We  consider for any $\varepsilon > 0$ the function $\psi_\varepsilon : [0, +\infty) \ni r \mapsto (r + \varepsilon)^{1/2}$. The function $\psi_\varepsilon$ being differentiable with positive derivative, we get
\begin{equation*}
	\frac{\mathrm{d}}{\mathrm{d}t}\psi_{\varepsilon}\left(u(t)\right) \leq 2k(t)\sqrt{u(t)}\psi_\varepsilon^\prime\left(u(t)\right).
\end{equation*}
Using the fact that $2r\psi_{\varepsilon}^\prime(r^2) = r(r^2 + \varepsilon)^{-1/2} \leq 1$, we use dominated convergence theorem to conclude the proof. 
\end{proof}
Such a lemma is classical and applying it to Equation \eqref{eq:before:le2} concludes the proof of Equation \eqref{eq:sqrt}. 
\end{proof}

\end{document}